\documentclass[10pt]{amsart}
\usepackage{hyperref}
\usepackage{pb-diagram}
\usepackage{amssymb}
\usepackage[margin=1in]{geometry}

\numberwithin{equation}{section}
\newtheorem{thm}[equation]{Theorem}
\newtheorem{lem}[equation]{Lemma}
\newtheorem{prp}[equation]{Proposition}

\theoremstyle{definition}
\newtheorem{df}[equation]{Definition}

\newtheorem{cor}[equation]{Corollary}

\theoremstyle{remark}
\newtheorem*{rem}{Remark}

\DeclareMathOperator{\map}{map}

\DeclareMathOperator{\Aut}{Aut}

\DeclareMathOperator{\Rep}{Rep}
\DeclareMathOperator{\IR}{IrrRep}
\DeclareMathOperator{\res}{res}

\DeclareMathOperator{\Ob}{Ob}
\DeclareMathOperator{\hocolim}{hocolim}

\DeclareMathOperator{\Fib}{Fib}
\DeclareMathOperator{\GL}{GL}
\DeclareMathOperator{\Sympl}{Sp}
\DeclareMathOperator{\Hom}{Hom}
\DeclareMathOperator{\Mor}{Mor}
\DeclareMathOperator{\Inn}{Inn}
\DeclareMathOperator{\Out}{Out}

\DeclareMathOperator{\Tor}{Tor}
\DeclareMathOperator{\tr}{tr}
 
 \def\Ab{\mathbf{Ab}}
 \def\Sp{\mathbf{Sp}}
 \def\HSp{\mathbf{HSp}}
 
\def\C{\mathcal{C}}
\def\Z{\mathbb{Z}}
\def\F{\mathbb{F}}

\def\sP{\mathcal{P}} 
 
\def\cR{\mathcal{R}}
\def\cC{\mathcal{C}}
\def\cP{\mathcal{P}}
\def\fG{\mathfrak{G}}
\def\fI{\mathfrak{I}}
 \def\fF{\mathfrak{F}}

 \def\Ad[#1]#2{\mathrm{Ad}^{#1}_{#2}}
 \def\hoc[#1]{\hocolim_{\C}^{(#1)}F}

\newcount\twoarrowsep
\makeatletter
\@namedef{dgo@dd}{\let\dg@VECTOR=\dg@twoarrowedvector}%

\newcount\dg@XSHIFT
\newcount\dg@YSHIFT
\def\dg@twoarrowedvector(#1,#2)#3{%
   \begingroup
   \dg@XTEMP=#1\relax\multiply\dg@XTEMP\m@ne\relax
   \dg@YTEMP=#2\relax\multiply\dg@YTEMP\m@ne\relax
   \dg@ZTEMP=#1\relax
   \ifnum\dg@ZTEMP<\z@
     \multiply\dg@ZTEMP\m@ne\relax \fi
   \ifnum\dg@YTEMP<\z@
     \advance\dg@ZTEMP by -\dg@YTEMP
   \else \advance\dg@ZTEMP by \dg@YTEMP \fi
   \dg@XSHIFT=#2\relax\multiply\dg@XSHIFT\m@ne\relax\multiply\dg@XSHIFT\twoarrowsep\relax
     \divide\dg@XSHIFT by \dg@ZTEMP\relax
   \dg@YSHIFT=#1\relax\multiply\dg@YSHIFT\twoarrowsep\relax\divide\dg@YSHIFT by \dg@ZTEMP\relax
   \begin{picture}(0,0)%
      \thinlines
      \put(-\dg@XSHIFT,-\dg@YSHIFT){\vector(#1,#2){#3}}%
      \put(\dg@XSHIFT,\dg@YSHIFT){\vector(#1,#2){#3}}
   \end{picture}%
   \endgroup}%
\makeatother

\title{Homotopy representations of the unitary groups}
\author[W. Lubawski]{Wojciech Lubawski}
\address{Theoretical Computer Science Department\\
Jagiellonian University\\ Go{\l}{e}bia 24\\ 00-300 Krak\'ow, Poland}
\address{Institute of Mathematics\\ Polish Academy of Sciences\\
\'Sniadeckich 8 \\ 00-956 Warszawa, Poland}
\email{w.lubawski@gmail.com}

\author[K. Ziemia\'nski]{Krzysztof Ziemia\'nski}
\address{Faculty of Mathematics, Informatics and Mechanics \\
University of Warsaw\\ 
Banacha  2 \\ 02-097 Warszawa, Poland}
\email{ziemians@mimuw.edu.pl}

\begin{document}

\begin{abstract}
	Let $G$ be a compact connected Lie group and let $\xi,\nu$ be complex vector bundles over the classifying space $BG$. The problem we consider is whether $\xi$ contains a subbundle which is isomorphic to $\nu$. The necessary condition is that for every prime $p$ the restriction $\xi|_{BN_p^G}$, where $N_p^G$ is a maximal $p$-toral subgroup of $G$, contains a subbundle isomorphic to $\nu|_{BN_p^G}$. We provide a criterion when this condition is sufficient, expressed in terms of $\Lambda^*$-functors of Jackowski, McClure \& Oliver and we prove that this criterion applies if $\nu$ is a universal bundle over $BU(n)$. Our result allows to construct new examples of maps between classifying spaces of unitary groups. While proving the main result, we develop the obstruction theory for lifting maps from homotopy colimits along fibrations, which generalizes the result of Wojtkowiak.
\end{abstract}

\maketitle
\section{Introduction}

Let us introduce the following property of complex vector bundles:

\begin{df}
	A vector bundle $\nu$ over a space $X$ has \emph{the splitting property with respect to a map $f:A\to X$} if the following holds for every bundle $\xi$ over $X$. If there exists a bundle $\xi'_A$ such that $f^*\xi\simeq f^*\nu\oplus \xi'_A$, then there exists a vector bundle $\xi'$ such that $\xi\simeq \nu\oplus\xi'$.
\end{df}

The subject of this paper is the splitting property of vector bundles in the special case. Let $G$ be a compact connected Lie group. For every prime $p$ there exists a unique up to conjugation maximal $p$-toral subgroup $N_p^G\subseteq G$. Its identity component is a maximal torus of $G$, and the group of its components is a $p$-Sylow subgroup of the Weyl group of $G$.

\begin{df}
	A vector bundle $\nu$ over $BG$ has \emph{the splitting property} if it has the splitting property with respect to a map
	\[
		\coprod_{\text{$p$ prime}} BN_p^G \to BG
	\]
	induced by the inclusions $N_p(G)\subseteq G$.
\end{df}

The splitting property in this particular case is especially useful due to the theorem of Notbohm \cite{N} which implies that isomorphism classes of complex vector bundles over $BN_p^G$ are in 1-1 correspondence with isomorphism classes of unitary representations of $N_p^G$. Jackowski \& Oliver \cite{O} proved that the trivial bundle over $BG$ has the splitting property. In the present paper we extend the result to the universal bundle $\gamma^n$ over the classifying space of the unitary group $BU(n)$ proving the following theorem:

\begin{thm}\label{t:MainBundle}
	The universal bundle $\gamma^n$ over the classifying space of the unitary group $BU(n)$ has the splitting property.
\end{thm}

Vector bundles over the classifying space can be interpreted and maps $BG \to BU(d)$ which will be called \emph{homotopy representations}. The natural source of homotopy representations are linear representations of $G$ i.e. homomorphisms $G \to  U(n)$. The vector bundles corresponding to maps induced by linear representations are the bundles associated with the universal $G$-bundle $EG \to BG$. Thus the above two properties of vector bundles can be easily formulated in terms of homotopy
theory. Since proofs will be given in those terms let us state them explicitly.

For a compact connected Lie group $G$ we will formulate the splitting property in terms of virtual characters. To every homotopy representation $f : BG \to BU(d)$ we associate a virtual character $\chi_f\in R(G)$. Let $T\subseteq G$ be a maximal torus of $G$. By the Notbohm theorem \cite{N} restriction $f|_{BT}: BT \to BU(d)$ is defined by a unique up to isomorphism linear representation $\varrho_f : T \to U(d)$ and its character $\chi_{\varrho_f}$ is invariant under the Weyl group action on the representation ring $R(T)$. Thus via the classical isomorphism $R(G) \simeq R(T)^W$ we can consider it as a virtual character $\chi_f\in R(G)$.
The virtual characters coming from homotopy representations we will call \emph{homotopy characters}. A natural question arises: when a virtual character $\alpha\in R(G)$ is a homotopy character? The Dwyer-Zabrodsky-Notbohm theorem provides the following constraint: for every prime $p$ and every $p$-toral subgroup $P\subseteq G$ the restriction of $\alpha$ to $P$ is the character of a linear representation. Such characters of the group $G$ will be called $\cP$-characters of $G$. Since $N_p$ is a maximal $p$-toral subgroup of $G$, a virtual character $\mu\in R(G)$ is a $\cP$-character if for every prime the restriction of $\mu$ to $N_p$ is a character of linear representation. Now we can formulate the splitting property for homotopy representations of the compact, connected Lie groups.

\begin{df}
	A homotopy character $\mu\in R(G)$ of a compact connected Lie group $G$ has \emph{the splitting property} if every $\mathcal{P}$-character $\nu$ of $G$ such that $\mu+\nu$ is a homotopy character is also a homotopy character.
\end{df}

Hence, Theorem \ref{t:MainBundle} can be reformulated as

\begin{thm}\label{t:Main}
	For every natural number $n$ the character $\iota$ of the identity representation $BU(n)\rightarrow BU(n)$ has the splitting property.
\end{thm}

Homotopy characters which have splitting property can be used to construct interesting new maps between classifying spaces. Description of homotopy classes of maps between classifying spaces of compact Lie groups is a classical topic of homotopy theory. The first examples of self maps not induced by homomorphisms were constructed by Sullivan \cite{Su} for classifying spaces $BU(n)$. For every integer $k$ prime to $n!$ there exists a map, known as the unstable Adams operation, $\Psi^k : BU(n)\to BU(n)$ such that restricted to the diagonal matrices is induced by $k$-power homomorphism. As was shown over twenty years later in \cite{JMO}, this condition characterizes such maps uniquely up to homotopy.
The general question how to describe the set of homotopy classes of maps $[BG,BH]$ in terms of groups $G$ and $H$ attracted attention of many researchers through decades, but still remains open. Our Theorem \ref{t:MainBundle} and \cite[Proposition 1.13]{JO} lead to construction of maps $BU(n) \to BU(d)$  which can not be produced by compositions of sums and tensor products of the unstable Adams operations and maps induced by homomorphisms. The crucial observation is that characters of some irreducible representations of $U(n)$ can be written as a nontrivial sum of $\cP$-characters $\mu+\nu$, where $\nu$ is the character of either the trivial or the identity representation. Then $\mu$ is a homotopy character. This construction applies to the classification of low-dimensional homotopy representations \cite{LZ}.

\subsection*{Criterion for splitting $\mathcal{P}$-characters}

Theorem \ref{t:Main} is a consequence of a more general criterion for splitting of $\sP$-characters. Before formulating it we need to introduce some definitions. Recall after \cite{JMO} that a $p$-toral group $P\subseteq G$ is \emph{$p$-stubborn} if $N_G(P)/P$ is a finite group and contains no non-trivial normal $p$-subgroups. Let $\Z_p^\wedge$ be the ring of $p$-adic integers, and for a finite group $\Gamma$ and a $\Z_p^\wedge[\Gamma]$-module $M$  let $\Lambda^i(\Gamma;M)$ be $\Lambda^*$-functors introduced in \cite[5.3]{JMO}  (see also \ref{e:DefLambda}).

For a compact Lie group $H$ let $\IR(H)$ be the set of isomorphisms classes of irreducible complex representations of $H$. For a representation $\alpha$ of $H$ and $\varrho\in\IR(H)$ let $c_\alpha^\varrho$ be the number of summands isomorphic to $\varrho$ in a decomposition of $\alpha$ into a sum of irreducible subrepresentations. Futhermore, let $\IR(P,\alpha)=\{\varrho\in\IR(P):\; c_\alpha^\varrho>0\}$. Any (left) action $N\to \Out(H)$ of a finite group $N$ on $H$ by outer automorphisms induces a right action on $\IR(H)$; furthermore, if $\alpha$ is $N$-invariant, then $\IR(P,\alpha)$ is a sub-$N$-set of $\IR(P)$.

For a $\mathcal{P}$-character $\xi$ of $G$ and a $p$-stubborn subgroup $P\subseteq G$ let $\xi_P$ be a representation having character $\res^{G}_P\xi$. Now we are ready to formulate the following:

\begin{thm}\label{t:MainCrit}
	Let $G$ be a compact connected Lie group and let $\mu$, $\nu$ be $\sP$-characters of $G$ such that $\mu+\nu$ is a homotopy character. Assume that for every prime $p$, every $p$-stubborn subgroup $P\subseteq G$ and every $N_G(P)/P$-orbit $X\subseteq \IR(P,\mu_P)\cap\IR(P,\nu_P)$ and $i\geq 3$ holds 
	\[
		\Lambda^{i}(N_G(P)/P; \Z_p^\wedge[X])=0.
	\]
	Then both $\mu$ and $\nu$ are homotopy characters.
\end{thm}

Theorem \ref{t:MainCrit} allows us to prove even stronger result than Theorem \ref{t:Main}, namely the splitting property for characters of unstable Adams operations (\ref{t:Adams}).

\subsection*{Lifting maps from homotopy colimits}
Let $\mathcal{C}$ be a small category, $F:\mathcal{C}\rightarrow \Sp$ a diagram of spaces and $X$ a space. Consider a collection of maps $\{f_c:F(c)\rightarrow X\}_{c\in\Ob(\mathcal{C})}$ which is homotopy compatible, i.e.\ such that for every morphism $c\buildrel{\alpha}\over \rightarrow c'$ maps $f_{c'}\circ F(\alpha)$ and $f_c$ are homotopic. According to results of Wojtkowiak \cite{W}, existence of an extension of $\coprod f_c:\coprod F(c)\rightarrow X$ to a map $f:\hocolim_{\C}F\rightarrow X$  depends on vanishing of obstructions lying in groups
\[
	H^{i+1}(\mathcal{C};\pi_i(\map(F(-),X)_{f_{(-)}})).
\]
In this paper we consider a generalization of this problem. Let $p:Y\rightarrow Z$ be a fibration. Fix a map $f:\hocolim_{\C}F\rightarrow Z$ and its partial lifting $\coprod g_c:\coprod F(c)\rightarrow Y$ (cf.\ \ref{e:HLDiag}) such that the collection $\{g_c\}_{c\in\Ob(\mathcal{C})}$ is homotopy compatible. We prove (\ref{t:Obstr}) that, under certain assumptions, this lifting can be extended to $\hocolim_{\C}F$ if groups
\[
	H^{i+1}(\mathcal{C}; \pi_i(\Fib(-)))
\]
vanish for $i>0$, where $\Fib:\mathcal{C}^{op}\to\HSp$ is a certain functor such that $\Fib(c)$ is a fiber of the fibration 
\[\map(F(c),Y)_{g_c}\rightarrow \map(F(c),Z)_{f|_{F(c)}}.\]

\subsection*{Acknowledgements}
The authors would like to thank Stefan Jackowski and the anonymous referee for many valuable suggestions and comments.

\section{Preliminaries}


For compact Lie groups $H,L$ let $R(H)$ be the unitary representation ring of $H$, \hbox{$R^+(H)\subseteq R(H)$} the semiring of isomorphism classes of unitary representations of $H$ and let $\Rep(H,L):=\Hom(H,L)/\Inn(L)$ be the set of representations of $H$ in $L$. If $H\subseteq L$, then $C_L(H)$ denotes the centralizer of $H$ in $L$. For a space $X$ we denote by $X_p^\wedge$ the $\F_p$-completion in the sense of Bousfield and Kan \cite{BK}.

\subsection*{Dwyer-Zabrodsky-Notbohm theorem}
An important tool we use in this paper is the following theorem due to Dwyer, Zabrodsky and Notbohm. 
\begin{thm}[\cite{N}]\label{t:DZN}
	Let $P$ be a $p$-toral group and $H$ a compact Lie group.
	\begin{itemize}
	\item{
		The map
		\[
			\Rep(P,H)\ni\varphi\mapsto [B\varphi]\in [BP,BH]
		\]
		is a bijection.
	}	
	\item{For every homomorphism $\alpha:P\to H$ the map
	\[
		ad_\alpha: BC_H(\alpha(P))\rightarrow \map(BP,BH)_{B\alpha},
	\]
	which is adjoint to the map induced by the multiplication homomorphism
	\[
		C_H(\varrho(P))\times P\ni (a,b)\mapsto \varrho(a)b \in  H
	\]
	is a mod $p$-equivalence, i.e.\ it induces an isomorphism in homology with $\Z/p$ as coefficients. In particular, the map
	\[(ad_\alpha)_p^\wedge:BC_H(\alpha(P))_p^\wedge\rightarrow (\map(BP,BH)_{B\alpha})_p^\wedge\]
	is a homotopy equivalence.}
	\end{itemize}
\end{thm}

\subsection*{$\mathcal{P}$-characters and homotopy characters}

Fix a compact connected Lie group $G$. Let $T\subseteq G$ be its maximal torus, and \hbox{$W\subseteq \Aut(T)$} its Weyl group. 

\begin{df}\label{d:HtpChar}
	\emph{A homotopy representation} of $G$ is a map $f:BG\rightarrow BU(d)$. We say that two homotopy representations are isomorphic if they are homotopic as maps. \emph{The character} $\chi(f)\in R(G)\cong R(T)^W$ of a homotopy representation $f$ is the character of the representation $\varrho:T\to U(d)$ such that $B\varrho\sim f|_{BT}$. By \ref{t:DZN}, such a character is determined uniquely.
\end{df}

\begin{df}
	A virtual character $\mu\in R(G)$ is \emph{a $\mathcal{P}$-character} if for every prime $p$ and every $p$-toral subgroup $P\subseteq G$ its restriction to $P$ is a character of a linear representation, i.e.\ $\mu|_{P}\in R^+(P)$. The set of $\mathcal{P}$-characters of $G$ will be denoted by $R_{\mathcal{P}}(G)$. Whenever it would not lead to confusion, we will denote by $\mu_P$ a unitary representation of $P$ having character $\mu|_P$.
\end{df}

\begin{prp}
	The character of a homotopy representation is a $\mathcal{P}$-character.
\end{prp}
\begin{proof}
	If $f:BG\to BU(d)$ is a homotopy representation and $P\subseteq G$ a $p$-subgroup, then \ref{t:DZN} implies that the restriction $f|_{BP}$ is induced by a unitary representation of $P$.
\end{proof}

\begin{df}
	A $\mathcal{P}$-character which is the character of a homotopy representation will be called \emph{a homotopy character}; the semiring of homotopy characters of $G$ will be denoted by $R_h(G)$.
\end{df}	

 Note that there is a sequence of inclusions
\begin{equation}
	R^+(G)\subseteq R_h(G)\subseteq R_{\mathcal{P}}(G)\subseteq  R(G)\cong R(T)^W\subseteq R(T).
\end{equation}

\subsection*{$p$-homotopy characters}

\begin{df}\label{d:pGenuine}
	Let $p$ be a prime integer. \emph{A $p$-homotopy representation} of $G$ is a map $f_p:BG\rightarrow BU(d)_p^\wedge$. We say that a $\mathcal{P}$-character $\mu\in R_{\mathcal{P}}(G)$ is a character of $f_p$ if the diagram
\[
	\begin{diagram}
		\node{BT}
			\arrow{e,t}{B\mu_T}
			\arrow{s,l}{B(T\subseteq G)}
		\node{BU(d)}
			\arrow{s,r}{(-)_p^\wedge}
	\\	
		\node{BG}
			\arrow{e,t}{f_p}
		\node{BU(d)_p^\wedge}
	\end{diagram}
\]
commutes. A $\cP$-character which is the character of a $p$-homotopy representation is called \emph{a $p$-homotopy character}.
\end{df}

\begin{prp}\label{p:PrimeHChar}
	A $\mathcal{P}$-character $\mu$ is a homotopy character if and only if it is a $p$-homotopy character for all primes $p$.
\end{prp}
\begin{proof}
	It is an immediate consequence of \cite[1.2]{JMOrev}.
\end{proof}

\subsection*{Subgroup homotopy decomposition}

Let $p$ be a prime integer.
\begin{df}
	A $p$-toral subgroup $P\subseteq G$ is \emph{$p$-stubborn} if $N_G(P)/P$ is a discrete group and contains no non-trivial normal $p$-subgroups. Let $\mathcal{R}_p(G)$ be the category of $G$-orbits having the form $G/P$ for a $p$-stubborn $P\subseteq G$ and $G$-maps.
\end{df}

\begin{thm}[{\cite[Theorem 1.4]{JMO}}]\label{t:JMODecomp}
	The map
	\begin{equation*}
		\varepsilon^p_G:\operatorname*{hocolim}_{G/P\in\mathcal{R}_p(G)} EG\times_G G/P\rightarrow EG\times_G *\cong BG
	\end{equation*}
	induced by projections $G/P\rightarrow *$ induces an equivalence on homology with $\F_p$-coefficients. In particular, the map
	\begin{equation*}
		(\varepsilon^p_G)_p^\wedge:\left(\operatorname*{hocolim}_{G/P\in\mathcal{R}_p(G)} EG\times_G G/P\right)_p^\wedge \rightarrow  BG_p^\wedge
	\end{equation*}	
	is a homotopy equivalence.
\end{thm}

\begin{rem}
	If $\mu$ is a $\mathcal{P}$-character of $G$ of dimension $d$, then the family
	\[
		EG\times_G G/P\simeq BP\xrightarrow{B\mu_P} BU(d)
	\]
	is a homotopy compatible family of maps from the decomposition functor $EG\times_G(-)$ to $BU(d)$.
\end{rem}

\subsection*{Functors $\Lambda^*$}
Let $p$ be a prime. Here we recall methods of calculating higher limits of certain functors on the category $\mathcal{R}_p(G)$ developed in \cite{JMO}.
For a commutative ring $R$ and a small category $\cC$ we define \emph{$R[\cC]$-modules} as contravariant functors from $\cC$ into the category of $R$-modules. If $F$ is an $R[\cC]$-module, then $H^i(\cC; F)$ denotes the $i$-th  right derived functor of the inverse limit of $F$ and will be referred to as the $i$-th coholomology group of $\cC$ with coeffcients $M$; this notion coincides with group cohomology if $\cC$ is a single-object category with invertible morphisms.

 For a finite group $\Gamma$ let $\mathcal{O}_p(\Gamma)$ be the category of $\Gamma$-orbits whose isotropy groups are $p$-groups and $\Gamma$-maps. For a $\Z_p^\wedge[\Gamma]$-module $M$ and $i\geq 0$ define groups
\begin{equation}\label{e:DefLambda}
	\Lambda^i(\Gamma;M):=H^i(\mathcal{O}_p(\Gamma); F_M^{\Gamma}),
\end{equation}
where $F^\Gamma_M$ is a $\Z_p^\wedge[\mathcal{O}_p(\Gamma)]$-module defined by
\begin{equation}
	F_M^\Gamma(\Gamma/P)=\begin{cases} M & \text{for $P=1$}\\ 0 & \text{otherwise.}\end{cases}
\end{equation}
These groups play an important role in calculating cohomology of $\Z_p^\wedge[\mathcal{R}_p(G)]$-modules. Their crucial property is the following:

\begin{prp}\label{p:LambdaCatVan}
	Let $G$ be a compact Lie group, $F$ a $\Z_p^\wedge[\mathcal{R}_p(G)]$-module and $r$ an integer. Assume that $\Lambda^i(N_G(P)/P;F(G/P))=0$ for all $i\geq r$ and all $G/P\in\mathcal{R}_p(G)$. Then $H^i(\mathcal{R}_p(G);F)=0$ for $i\geq r$.
\end{prp}
\begin{proof}
	This follows from \cite[1.8]{JMO} and \cite[1.10.(ii)]{JMOrev}.
\end{proof}

Let us recall several properties of functors $\Lambda$ proven in \cite{JMO}:
\begin{prp}\label{p:LambdaProperties}
	Let $\Gamma$ be a finite group and let $M$ be a $\Z_p^\wedge[\Gamma]$-module.
	\begin{enumerate}
	\item{If $p$ divides the order of $\Gamma$, then $\Lambda^0(\Gamma;M)=0$. Otherwise $\Lambda^0(\Gamma;M)=M^\Gamma$.}
	\item{Let $H=\ker(\Gamma\to \Aut(M))$. If $p$ divides the order of $H$, then $\Lambda^*(\Gamma;M)=0$; otherwise $\Lambda^*(\Gamma;M)=\Lambda^*(\Gamma/H;M)$.}
	\end{enumerate}
\end{prp}
\begin{proof}
	This follows from \cite[6.1]{JMO}.
\end{proof}

In the following proposition all tensor products are taken with respect to the ring $\Z_p^\wedge$.

\begin{prp}\label{p:LambdaExtLemma}
	Let $\Gamma$ be a finite group, $M$ a finitely generated $\Z_p^\wedge$-module, $X$ a finite $\Gamma$-set and $r\geq 0$ an integer. Assume that $\Lambda^i(\Gamma;\Z_p^\wedge[X])=0$ for every $i\geq r$. Then $\Lambda^i(\Gamma;M\otimes \Z_p^\wedge[X])=0$ for $i\geq r$.
\end{prp}
\begin{proof}
	By \cite[6.1.(v)]{JMO} there is an exact sequence
	\[
	0 \to \bigoplus_{k+l=i+1} \Tor(\Lambda^k(\Gamma;\Z_p^\wedge[X]),\Lambda^l(1;M))\to \Lambda^i(\Gamma;\Z_p^\wedge[X]\otimes M)\to
	\bigoplus_{k+l=i}\Lambda^{k}(\Gamma;\Z_p^\wedge)\otimes \Lambda^l(1;M)\to 0
	\]
	which by \cite[6.1.(i)]{JMO} reduces to
	\[
	0 \to \Tor(\Lambda^{i+1}(\Gamma;\Z_p^\wedge[X]),M)\to \Lambda^i(\Gamma;\Z_p^\wedge[X]\otimes M)\to
	\Lambda^{i}(\Gamma;\Z_p^\wedge)\otimes M\to 0		
	\]
	By assumption $\Lambda^{i}(\Gamma;\Z_p^\wedge[X])=\Lambda^{i+1}(\Gamma;\Z_p^\wedge[X])=0$. The conclusion follows.
\end{proof}

\section{Guide to the argument}

In this Section we prove Theorem \ref{t:MainCrit} and then deduce Theorem \ref{t:Main}. Let $G$ be a compact connected Lie group, and $T\subseteq G$ its maximal torus. Let $\mu$ and $\nu$ be $\mathcal{P}$-characters of $G$ having dimensions $d$ and $d'$ respectively and assume that $\mu+\nu$ is a homotopy character. Our goal is to check that, under certain assumptions, both $\mu$ and $\nu$ are homotopy characters. By \ref{p:PrimeHChar} it is sufficient to prove that they are $p$-homotopy characters for all primes $p$. Fix a prime $p$ and a $p$-homotopy representation $f:BG\to BU(d+d')_p^\wedge$ having character $\mu+\nu$. We will prove a sequence of criteria (\ref{p:Sketch}, \ref{p:MTPLift}, \ref{t:Crit}) which imply that $\mu$ and $\nu$ are $p$-homotopy characters.

\subsection*{A lifting extension problem}
The first step in the proof of Theorem \ref{t:MainCrit} is the following observation:

\begin{prp}\label{p:Sketch}
	If there exists a map $g:BG\to (BU(d)\times BU(d'))_p^\wedge$ such that the diagram
	\[
		\begin{diagram}
			\node{BT}
				\arrow[3]{e,t}{(B\mu_T\times B\nu_T)_p^\wedge}
				\arrow{s,l}{B(T\subseteq G)}
			\node{}
			\node{}
			\node{(BU(d)\times BU(d'))_p^\wedge}
				\arrow{s,r}{\oplus}
		\\
			\node{BG}
				\arrow[3]{e,t}{f}
				\arrow{ene,t}{g}
			\node{}
			\node{}
			\node{BU(d+d')_p^\wedge}
		\end{diagram}
	\]
	commutes up to homotopy, then $\mu$ and $\nu$ are $p$-homotopy characters.
\end{prp}
\begin{proof}
	The restriction of the composition
	\[
		f_1:BG\xrightarrow{g}(BU(d)\times BU(d'))_p^\wedge\xrightarrow{proj} BU(d)_p^\wedge
	\]
	to $BT$ is homotopic to $(B\mu_T)_p^\wedge$. Therefore, $f_1$ is a $p$-homotopy representation having character $\mu$. A similar argument applies to $\nu$.
\end{proof}

\subsection*{A homotopy decomposition}
We will use the homotopy decomposition \ref{t:JMODecomp} to find criteria for existence a lifting extension in the diagram (\ref{p:Sketch}). Consider the diagram
\begin{equation}\label{e:LiftingGFull}
	\begin{diagram}
		\node{\coprod_{G/P\in\mathcal{R}_p(G)}EG\times_G G/P}
			\arrow[2]{e,t}{\coprod B(\mu_P\times \nu_P)_p^\wedge}
			\arrow[2]{s,l}{\subseteq}
			\arrow{ese,t,..}{\coprod g_P}
		\node{}
		\node{(BU(d)\times BU(d'))_p^\wedge}
			\arrow{s,r}{\cong}
	\\
		\node{}
		\node{}
		\node{(EU(d+d')/U(d)\times U(d'))_p^\wedge}
			\arrow{s,r}{q}
	\\
		\node{\operatorname*{hocolim}_{G/P\in\mathcal{R}_p(G)} EG\times_G G/P}
			\arrow{e,t}{\varepsilon^p_G}
			\arrow{ene,t,..}{g}
		\node{BG}
			\arrow{e,t}{f}
		\node{BU(d+d')_p^\wedge}
	\end{diagram}
\end{equation} 
where $q$ is induced by the identity map on $EU(d+d')$, and $\varepsilon^p_G$ is a mod $p$-equivalence from Theorem \ref{t:JMODecomp}. The diagram of solid arrows commutes up to homotopy. Since $q$ is a fibration, we can choose maps
\[
	g_P:EG\times_G G/P\to (EU(d+d')/U(d)\times U(d'))_p^\wedge 
\]
 such that the upper triangle commutes up to homotopy, and the remaining part of the diagram, i.e.\ 
\begin{equation}\label{e:LiftingGStrict}
\begin{diagram}
	\node{\coprod_{G/P\in\mathcal{R}_p(G)}EG\times_G G/P}
		\arrow{e,t}{\coprod g_P}
		\arrow{s,l}{\subseteq}
	\node{(EU(d+d')/U(d)\times U(d'))_p^\wedge}
		\arrow{s,r}{q}
\\
	\node{\operatorname*{hocolim}_{G/P\in\mathcal{R}_p(G)} EG\times_G G/P}
			\arrow{e,t}{f\circ \varepsilon^p_G}
			\arrow{ne,t,..}{g}
	\node{BU(d+d')_p^\wedge}
\end{diagram}
\end{equation}
commutes strictly.

\begin{prp}\label{p:MTPLift}
	If the lifting extension problem (\ref{e:LiftingGStrict}) has a solution, then $\mu$ and $\nu$ are $p$-homotopy characters.
\end{prp}
\begin{proof}
	Let $g$ be a lifting extension. Since $\varepsilon^p_G$ is a $p$-equivalence, there exists a unique up to homotopy map $f'$ such that the composition
	\[	
		\operatorname*{hocolim}_{G/P\in\mathcal{R}_p(G)} EG\times_G G/P\xrightarrow{\varepsilon^p_G} BG \xrightarrow{f'}(EU(d+d')/U(d)\times U(d'))_p^\wedge 
	\]
	is homotopic to $g$. The composition
	\[
		\tilde{f}:BG\xrightarrow{f'}(EU(d+d')/U(d)\times U(d'))_p^\wedge \xrightarrow{\simeq} (BU(d)\times BU(d'))_p^\wedge
	\]
	 fits into the diagram (\ref{p:Sketch}) since $T$ is contained in some $p$-stubborn subgroup. Thus the conclusion follows from \ref{p:Sketch}.
\end{proof}

\subsection*{Obstruction theory}

In order to construct a lifting extension of the diagram (\ref{e:LiftingGStrict}) we will apply results obtained in Section \ref{s:Obstr}. Let us sketch the argument we use here; the formal proof is given in Sections \ref{s:HomotopyActions} and \ref{s:Fibers}. Let $\Fib_{\mu,\nu}^p:\cR_p(G)\to \Sp$ be the functor defined in \ref{e:DefFib} which is associated to this lifting extension problem. We need to check that the assumptions which allow to construct this functor are satisfied but it will become clear when we calculate the homotopy type of the spaces $\Fib_{\mu,\nu}^p(G/P)$. By Theorem \ref{t:Obstr}, the problem (\ref{e:LiftingGStrict}) has a solution if 
\[
	H^{i+1}(\cR_p(G);\pi_i\Fib^p_{\mu,\nu})=0
\]
for all $i\geq 1$. By \ref{p:LambdaCatVan} it is sufficient to prove that
\[
	\Lambda^{i+1}(N_G(P)/P; \pi_{i}\Fib^p_{\mu,\nu}(G/P))=0
\]
for $G/P\in\mathcal{R}_p(G)$, $i\geq 1$. This makes important to describe the spaces  $\Fib_{\mu,\nu}^p(G/P)$ and the action of the groups $N_G(P)/P$ on them. For a $p$-stubborn subgroup $P\subseteq G$ the space $\Fib_{\mu,\nu}^p(G/P)$ is the homotopy fiber of the map 
\[
	\map(EG\times G/P,(BU(d)\times BU(d'))_p^\wedge)_{B\mu_P\times B\nu_P}\to \map(EG\times G/P,BU(d+d')_p^\wedge)_{B(\mu_P\oplus\nu_P)}
\]
induced by the inclusion $U(d)\times U(d')\subseteq U(d+d')$. By Dwyer-Zabrodsky-Notbohm Theorem (\ref{t:DZN}) this map is homotopy equivalent to
\[
	BC_{U(d)\times U(d')}(\mu_P(P))_p^\wedge\times BC_{U(d)\times U(d')}(\nu_P(P))_p^\wedge \to BC_{U(d+d')}((\mu_P\times\nu_P)(P))_p^\wedge,
\]
which is by Schur's lemma equivalent to the map
\[
	\prod_{\varrho\in\IR(P)} B(U(c_{\mu_P}^\varrho)\times U(c_{\nu_P}^\varrho))_p^\wedge \to \prod_{\varrho\in\IR(P)}BU(c_{\mu_P}^\varrho+c_{\nu_P}^\varrho)_p^\wedge,
\]
induced again by inclusions $U(c_{\mu_P}^\varrho)\times U(c_{\nu_P}^\varrho)\subseteq U(c_{\mu_P}^\varrho+c_{\nu_P}^\varrho)$.
Thus there is a homotopy equivalence
\begin{equation}\label{e:FibGEq}
	\Fib_{\mu,\nu}^p(G/P)\simeq\prod_{\varrho\in\IR(P)} \left( U(c_{\mu_P}^\varrho+c_{\nu_P}^\varrho)/(U(c_{\mu_P}^\varrho)\times U(c_{\nu_P}^\varrho)\right)_p^\wedge.
\end{equation}
To proceed with calculations we need to describe also the action of the group $\Aut_{\mathcal{R}_p(G/P)}(G/P)=N_G(P)/P$ on $\Fib_{\mu,\nu}^p(G/P)$. There is an $N_G(P)/P$-action on the right-hand side of the equivalence (\ref{e:FibGEq}) induced by the action on $\IR(P)$. However, it is far from clear that these actions coincide.

For a $p$-stubborn group $P\subseteq G$ define a left $N_G(P)/P$-space
\begin{equation}\label{e:DefFrakG}
	\fG_{\mu_P,\nu_P}:=\prod_{\varrho\in\IR(P)}\left( U(c_{\mu_P}^\varrho+c_{\nu_P}^\varrho)/(U(c_{\mu_P}^\varrho)\times U(c_{\nu_P}^\varrho)\right)_p^\wedge,
\end{equation}
where $N_G(P)/P$ acts by permuting factors, i.e.\ $\gamma\left((x_\varrho)_{\varrho\in\IR(P)}\right)=(x_{\gamma^*\varrho})_{\varrho\in\IR(P)}$ for $\gamma\in N_G(P)/P$. We say that $\Gamma$-spaces $X$ and $Y$ are homotopy $\Gamma$-equivalent (cf.\ \ref{d:hNActions}) if there exists a map $f:X\to Y$ such that the diagram
\begin{equation}
	\begin{diagram}
		\node{X}
			\arrow{e,t}{f}
			\arrow{s,l}{\gamma}
		\node{Y}
			\arrow{s,r}{\gamma}
	\\
		\node{X}
			\arrow{e,t}{f}
		\node{Y}
	\end{diagram}
\end{equation}
commutes up to homotopy for every $\gamma\in \Gamma$.

\begin{thm}\label{t:hNEq}
	For every $p$-stubborn subgroup $P\subseteq G$ the $N_G(P)/P$-spaces $\fG_{\mu_P,\nu_P}$ and $\Fib_{\mu,\nu}^p(G/P)$ are homotopy $N_G(P)/P$-equivalent.
\end{thm}

The proof of this Theorem is postponed to Sections \ref{s:HomotopyActions} and \ref{s:Fibers}. This result allows us to prove the following

\begin{thm}\label{t:Crit}
	If for every $p$-stubborn subgroup $P\subseteq G$ and every $i\geq 2$ holds 
	\[
		\Lambda^{i+1}(N_G(P)/P;\pi_i(\mathfrak{G}_{\mu_P,\nu_P}))=0,
	\]
	then $\mu$ and $\nu$ are $p$-homotopy characters.
\end{thm}
\begin{proof}
	We will apply Theorem \ref{t:Obstr} to prove the existence of the lifting extension of the diagram (\ref{e:LiftingGStrict}). First, we need to check that assumptions of \ref{t:Obstr} are satisfied. For every $p$-stubborn subgroup $P\subseteq G$ the group $C_{U(d+d')}((\mu_P\times \nu_P)(P))$ is connected (by Schur's lemma), and  there is a homotopy equivalence
\[
	\map(EG\times_G G/P, BU(d+d')_p^\wedge)_{B(\mu_P\times\nu_P)}\simeq BC_{U(d+d')}((\mu_P\times \nu_P)(P))_p^\wedge
\]
(by  Dwyer-Zabrodsky-Notbohm Theorem \ref{t:DZN}). Thus, $\map(EG\times_G G/P, BU(d+d')_p^\wedge)_{B(\mu_P\times\nu_P)}$ is simply connected. Furthermore, the space $\Fib_{\mu,\nu}^p(G/P)$  is simply connected since it is homotopy equivalent to $\mathfrak{G}_{\mu_P,\nu_P}$ (by \ref{t:hNEq}). Again by \ref{t:hNEq} and the assumption 
	\[
		\Lambda^{i+1}(N_G(P)/P; \Fib_{\mu,\nu}^p(G/P))=\Lambda^{i+1}(N_G(P)/P; \pi_i(\mathfrak{G}_{\mu_P,\nu_P}))=0
	\]
	for $i\geq 2$ and $G/P\in\cR_p(G)$. The space $\fG_{\mu_P,\nu_P}$ is simply connected, this equation is valid also for $i=1$. Thus, by \ref{p:LambdaCatVan}
	\[
		H^{i+1}(\cR_p(G);\Fib_{\mu,\nu}^p)=0
	\]
	for $i\geq 1$. Theorem \ref{t:Obstr} implies that there exists a lifting extension in the diagram (\ref{e:LiftingGStrict}). The conclusion follows from \ref{p:MTPLift}.
\end{proof}

\subsection*{Calculation of groups $\Lambda^*$}

\begin{prp}\label{p:LastLemma}
	Let $P\subseteq G$ be a $p$-stubborn subgroup. Assume that for every $N_G(P)/P$-orbit $X\subseteq \IR(P,\mu_P)\cap \IR(P,\nu_P)$ holds 
	\[
		\Lambda^{i+1}(N_G(P)/P; \Z_p^\wedge[X])=0
	\]
	for $i\geq 2$. Then $\Lambda^{i+1}(N_G(P)/P;\pi_i(\fG_{\mu_P,\nu_P}))=0$ for $i\geq 2$.
\end{prp}
\begin{proof}
	Denote $N=N_G(P)/P$,  $\fI:=\IR(P,\mu_P)\cap \IR(P,\nu_P)$ and
	\[
		Y_\varrho:=\left(U(c_{\mu_P}^\varrho+c_{\mu_P}^\varrho)/(U(c_{\mu_P}^\varrho)\times U(c_{\nu_P}^\varrho))\right)_p^\wedge
	\]
	Note that for $\varrho\in \IR(P)\setminus \fI$ either $c^\varrho_{\mu_P}=0$ or $c_{\nu_P}^\varrho=0$, thus
	\[
		\fG_{\mu_p,\nu_p}=\prod_{\varrho\in\fI}Y_\varrho.
	\]
	Let $\fI=X_1\cup\dots\cup X_k$ be a presentation as a sum of disjoint orbits and choose representatives $\varrho_j\in X_j$. Obviously $Y_\varrho=Y_{\varrho'}$ if $\varrho$ and $\varrho'$ lie in the same orbit. We have a sequence of isomorphisms of $\Z_p^\wedge[N_G(P)/P]$-modules:
	\[
		\pi_i\left(\fG_{\mu_p,\nu_p}\right)=\pi_i\left(\prod_{\varrho\in \fI} Y_\varrho\right) =\prod_{j=1}^k \pi_i\left(\prod_{\varrho\in X_j} Y_\varrho\right)
		= \prod_{j=1}^k \pi_i\left( Y_{\varrho_j}^{X_j}\right)
		= \prod_{j=1}^k \pi_i\left(Y_{\varrho_j}\right)\otimes \Z_p^\wedge[X_j].
	\]
	Since $\pi_i\left(Y_{\varrho_j}\right)$ is a finitely generated $\Z_p^\wedge$-module, by \ref{p:LambdaExtLemma} we have
	\[
		\Lambda^{i}(N;\pi_i(\fG_{\mu_P,\nu_P}))=\Lambda^{i}(N;\prod_{j=1}^k \Z_p^\wedge[X_j]\otimes \pi_i\left(Y_{\varrho_j}\right) ))
		=\bigoplus_{j=1}^k \Lambda^{i}(N; \Z_p^\wedge[X_j]\otimes \pi_i\left(Y_{\varrho_j}\right) ))=0.\qedhere
	\]
\end{proof}

\begin{proof}[Proof of {\ref{t:MainCrit}}]
	This follows immediately from \ref{t:Crit} and \ref{p:LastLemma}.
\end{proof}

\subsection*{Splitting of identity representation and Adams operations}

Finally, we will prove that Theorem \ref{t:MainCrit} applies in the case when $\nu=\psi^k$ is the character of the unstable Adams operation of index $k$. The special case is the character of the identity representation $\iota=\psi^1$.  Its proof depends on the following two propositions. For a finite set $X$ let $\Sigma_X$ denote the group of permutations of $X$.

\begin{prp}\label{p:EpiSigma}
	Let $P$ be a $p$-stubborn subgroup of $U(n)$. For every $N_{U(n)}(P)/P$-orbit $X\subseteq \IR(P,\psi^k_P)$ the map $N_{U(n)}(P)/P\to \Sigma_X$ induced by this action is an epimorphism.
\end{prp}

\begin{prp}\label{p:LambdaSigma}
	If $X$ is a finite set, then $\Lambda^i(\Sigma_X;\Z_p^\wedge[X])=0$ for $i\geq 2$.
\end{prp}

These propositions are main results of Sections \ref{s:RepresentationsOfStubbornSubgroups} and \ref{s:FunctorsLambdaSigma} respectively.

\begin{thm}\label{t:Adams}
	For every natural number $n$ and $k$ such that $(k,n!)=1$ the character $\psi^k$ of the unstable Adams operation $\Psi^k:BU(n)\rightarrow BU(n)$ has the splitting property.
\end{thm}
\begin{proof}
	Let $\mu$ be a $\cP$-character of $U(n)$ such that $\mu+\psi^k$ is a homotopy character. Fix a prime $p$, a $p$-stubborn subgroup $P\subseteq U(n)$ and  an $N_{U(n)}(P)/P$-orbit $X\subseteq \IR(P,\psi^k)$. By \ref{p:LambdaProperties}.(2), if $p$ divides the order of $\ker(N_{U(n)}(P)/P\to \Sigma_X)$, then $\Lambda^i(N_{U(n)}(P)/P;\Z_p^\wedge[X])=0$; otherwise
	\[
		\Lambda^i(N_G(P)/P;\Z_p^\wedge[X])=\Lambda^{i}(\Sigma_X; \Z_p^\wedge[X])=0
	\]
	for $i\geq 2$ by \ref{p:EpiSigma} and \ref{p:LambdaSigma}. Now the conclusion follows from \ref{t:MainCrit}.	
\end{proof}

Theorem \ref{t:Main} follows as a special case of the above.

\section{Obstruction theory} 
\label{s:Obstr}

In this section we develop the obstruction theory for lifting maps from homotopy colimits. This is an extension  of results obtained by Wojtkowiak \cite{W}. Let $\C$ be a small category, $F:\C\rightarrow \Sp$ a diagram of spaces and $p:Y\rightarrow Z$ a fibration. Fix a map $f:\hocolim_{\C}\; F\rightarrow Z$ and a family of maps $g_c:F(c)\rightarrow Y$, $c\in\Ob(\C)$ which represents a cocone in the homotopy category, i.e.\ an element
\[
	\{[g_c]\}_{c\in\Ob(\mathcal{C})}\in\lim_{c\in{\mathcal{C}}}\;[F(c),Y].
\]
Assume that the diagram of solid arrows
\begin{equation}\label{e:HLDiag}
	\begin{diagram}
		\node{\coprod_{c\in\C} F(c)}
			\arrow{e,t}{\coprod g_c}
			\arrow{s,l}{\subseteq}
		\node{Y}
			\arrow{s,r}{p}
	\\
		\node{\hocolim_{\C}\; F}
			\arrow{e,t}{f}
			\arrow{ne,t,..}{\bar{g}}
		\node{Z}
	\end{diagram}
\end{equation}
strictly commutes. For every $c\in\Ob(\C)$ denote $f_c:=f|_{F(c)}$ and
\begin{equation}\label{e:DefFib}
	\Fib(c):=(p_c)^{-1}(f_c),
\end{equation}
where
\[
	p_c:=\map(F(c),p):\map(F(c),Y)_{g_c}\xrightarrow{p_*}\map(F(c),Z)_{f_c}.
\]
Obviously $p_c$ is a fibration with fiber $\Fib(c)$. We will make two more assumptions: for every $c\in\Ob(\C)$
\begin{itemize}
\item{the connected component of the mapping space $\map(F(c),Z)_{f_c}$ is simply connected,}
\item{$\Fib(c)$ is a simple space (i.e.\ its fundamental group acts trivially on all homotopy groups).}
\end{itemize}

Every morphism $\alpha:c\rightarrow c'$ in $\C$ choose a map
\[
	\Fib(\alpha):\Fib(c')=p_{c'}^{-1}(f_{c'})\xrightarrow{F(\alpha)^*} p_c^{-1}(f_{c'}\circ F(\alpha))\xrightarrow{\simeq} p_c^{-1}(f_c)=\Fib(c).
\]
The first assumption guarantees that the homotopy class of $\Fib(\alpha)$ does not depend on the choice of a homotopy equivalence between fibers of $p_c$. Thus, $\Fib$ is a contravariant functor from $\mathcal{C}$ to the homotopy category $\HSp$. By the second assumption, for every $c\in \Ob(\C)$ homotopy groups of $\Fib(c)$ for different choices of basepoints are naturally isomorphic. Therefore the composition  $\pi_n\circ \Fib$ is a contravariant functor from $\mathcal{C}$ into the category of abelian groups $\Ab$. The main result of this Section is the following

\begin{thm}\label{t:Obstr}
	If $H^{i+1}(\mathcal{C};\pi_i(\Fib))=0$ for all $i>0$, then there exists a lifting extension \hbox{$\bar{g}:\hocolim_{\C}\; F\rightarrow Y$} which makes the diagram \ref{e:HLDiag} commutative.
\end{thm}

This results reduces to the result of \cite{W} in case where $Z$ is a single-point space.

\subsection*{Elementary lifting extension problem}
To establish notation, we recall some elementary results of the classical obstruction theory. Fix a fibration \hbox{$q:E\rightarrow B$} such that its fiber $F$ is simple, and assume that $B$ is simply connected. Consider the following lifting extension problem:
\begin{equation}\label{e:LiftingDiag}
\begin{diagram}
		\node{\partial\Delta^{n+1}}
			\arrow{e,t}{s}
			\arrow{s,l}{\subseteq}
		\node{E}
			\arrow{s,r}{q}
		\\
		\node{\Delta^{n+1}}
			\arrow{e,t}{r}
			\arrow{ne,t,..}{t}
		\node{B}
	\end{diagram}
\end{equation}
Let 
\begin{equation}
	\mathcal{L}(s,r):=\{t\in\map(\Delta^{n+1},E):\; t|_{\partial\Delta^{n+1}}=s\;\text{and}\; q\circ t=r\}
\end{equation}
be the space of  lifting extensions.

The map $s$ factors through the total space $r^*E$ of the pull-back fibration over $B$. Let
\begin{equation}\label{e:ElementaryO}
	\mathfrak{o}_n(s,r)\in [\partial\Delta^{n+1},r^*E]\cong \pi_n(r^*E)\simeq \pi_{n}(F)
\end{equation}
be the homotopy class corresponding to this factorization; we will call it \emph{the obstruction class}. This definition depends neither on the choice of basepoints of the homotopy groups (due to simplicity of $F$), nor on the choice of homotopy equivalence $r^*E\simeq F$ (since $B$ is simply connected).
For $t,t'\in\mathcal{L}(s,r)$ define \emph{the difference class} 
\begin{equation}\label{e:ElementaryD}
	\mathfrak{d}_{n+1}(t,t')\in [\Delta^{n+1}\cup_{\partial\Delta^{n+1}}\Delta^{n+1},r^*E]\simeq \pi_{n+1}(F),
\end{equation}
which is the obstruction class of the lifting extension problem
\begin{equation}\label{e:LiftingDiagDiff}
\begin{diagram}
		\node{\Delta^{n+1}\cup_{\partial\Delta^{n+1}}\Delta^{n+1}}
			\arrow{e,t}{t\cup t'}
			\arrow{s,l}{pr}
		\node{E}
			\arrow{s,r}{q}
		\\
		\node{\Delta^{n+1}}
			\arrow{e,t}{r}
			\arrow{ne,t,..}{ }
		\node{B}
	\end{diagram}
\end{equation}

Let $d_i:\Delta^{k}\rightarrow \Delta^{k+1}$ be the inclusion onto the $i$-th face of the simplex, and let $\Delta^k_{(l)}$ be the $l$-th skeleton of $\Delta^k$. Let us state some elementary properties of concepts defined above:
\begin{prp}\label{p:ObstrE} 
	Fix a map $r:\Delta^{n+1}\rightarrow B$.
	\begin{enumerate}
	\item{If $s:\partial\Delta^n\rightarrow E$ is a lifting of $r|_{\partial\Delta^n}$ along $q$, then $\mathcal{L}(s,r)\neq \emptyset$ if and only if $\mathfrak{o}_n(s,r)=0$.}
	\item{Let $s,s':\partial\Delta^{n+1}\rightarrow Y$ be liftings of $r|_{\partial\Delta^n}$ along $q$. Assume that $s|_{\Delta^{n+1}_{(n-1)}}=s'|_{\Delta^{n+1}_{(n-1)}}$. Then
		\[\mathfrak{o}_n(s',r)=\mathfrak{o}_n(s,r)+\sum_{i=0}^{n+1} (-1)^i \mathfrak{d}_n(s'\circ d_i, s\circ d_i).\]
	}
	\item{For any commutative diagram
\begin{equation*}
\begin{diagram}
		\node{\Delta^{n+2}_{(n)}}
			\arrow{e,t}{s}
			\arrow{s,l}{\subseteq}
		\node{E}
			\arrow{s,r}{q}
		\\
		\node{\Delta^{n+2}_{(n+1)}}
			\arrow{e,t}{r}
		\node{B}
	\end{diagram}
\end{equation*}	
holds $\sum_{i=0}^{n+2}(-1)^i \mathfrak{o}_n(s\circ d_i, r\circ d_i)=0$.
	}
	\item{
		If $s:\partial\Delta^n\rightarrow E$ is a lifting of $r|_{\partial\Delta^n}$ along $q$, then for every $t\in\mathcal{L}(s,r)$ and every $u\in\pi_{n+1}(F)$ there exists $t'\in\mathcal{L}(s,r)$ such that $\mathfrak{d}_{n-1}(t',t)=u$.\qed
	}
	\end{enumerate}
\end{prp}

\begin{prp}[Functoriality of obstruction and difference classes]\label{p:ObstrF}
	Consider a commutative diagram
\begin{equation*}
\begin{diagram}
		\node{}
		\node{F}
			\arrow{e,t}{c_F}
			\arrow{s}
		\node{F'}
			\arrow{s}
	\\
		\node{\partial\Delta^{n+1}}
			\arrow{s,l}{\subseteq}
			\arrow{e,t}{s}
		\node{E}
			\arrow{e,t}{c_E}
			\arrow{s,l}{q}
		\node{E'}
			\arrow{s,l}{q'}
	\\
		\node{\Delta^{n+1}}
			\arrow{e,t}{r}
		\node{B}
			\arrow{e,t}{c_B}
		\node{B'}
\end{diagram}
\end{equation*}
	where $q':E'\rightarrow B'$ is a fibration with a fibre $F'$. Assume that $B'$ is simply connected and that $F'$ is simple. Then
	\begin{enumerate}
	\item{$\mathfrak{o}_n(c_E\circ s,c_B\circ r)=(c_F)_*(\mathfrak{o}_n(s,r)).$}
	\item{If $t,t'\in\mathcal{L}(s,r)$ then
	\[
		\mathfrak{d}_{n+1}(c_E\circ t,c_E\circ t')=(c_F)_*\mathfrak{d}_{n+1}(t,t').\qed
	\]}
	\end{enumerate}
\end{prp}

\subsection*{Cochain complex}
We return to considering the lifting extension problem \ref{e:HLDiag}. Here we define a cochain complex which can be used to calculating cohomology groups $H^j(\mathcal{C};\pi_i(\Fib))$. Also obstruction classes and difference classes which are analogues of \ref{e:ElementaryO} and \ref{e:ElementaryD} will be defined as its cochains.

Let $N(\C)$ denote the nerve of $\C$ and let $N(\C)_i$ be the set of $i$-simplices of $N(\C)$. For $i,j>0$ define groups
\begin{equation}
	C^j_i=\prod_{\sigma\in N(\C)_j} \pi_i(\Fib(\sigma(0)))
\end{equation}
and differentials $\delta^j_i:C^j_i\rightarrow C_i^{j+1}$
\begin{equation}
	\delta_i^j(u)(\sigma)=\Fib(\sigma(0\rightarrow 1))_*(u(d_0\sigma))+\sum_{k=1}^{j+1} (-1)^k u(d_k(\sigma))
\end{equation}
for $u\in C^j_i$, $\sigma\in N(\mathcal{C})_{j+1}$. Let $Z^j_i, B^j_i,H^j_i$ denote respectively the cocyles, the coboundaries and the cohomology of the cochain complex $(C^*_i,\delta^*_i)$. Following \cite[Lemma 2]{O2}, $H^j(\mathcal{C};\pi_i(\Fib))=H^j_i$ for all $i$,$j$.

\subsection*{Adjoint maps}
Denote for short $X=\hocolim_{\mathcal{C}}\;F$ and $X_k=\hocolim^{(k)}_{\mathcal{C}}\;F$. For a space $W$, a map $a:X_k\rightarrow W$ and a simplex $\sigma\in N(\C)_n$ let
\begin{equation}
	\Ad[a]\sigma:\Delta^n_{(k)}\rightarrow \map(F(\sigma(0),W))
\end{equation}
be the adjoint map to the composition
\begin{equation*}
	\Delta^n_{(k)}\times F(\sigma(0))\xrightarrow{\sigma} \hoc[k]=X_k\buildrel{a}\over\rightarrow W.
\end{equation*}
Note that 
\begin{equation}\label{e:AdComp}
	\Ad[a]{\sigma}\circ d_i=\begin{cases}
		\Ad[a]{d_i\sigma} & \text{for $i>0$}\\
		F(\sigma(0\rightarrow 1))\circ \Ad[a]{d_0\sigma} & \text{for $i=0$,}
	\end{cases}
\end{equation}
and that maps $a:X_k\rightarrow W$ are in 1-1 correspondence with families of maps $\{\Ad[a]\sigma\}_{\sigma\in N(\C)}$ satisfying these relations.

\subsection*{Lifting extension spaces}
For $n\geq 0$ let $\mathcal{E}_0^n$ be the space of maps $g:X_n\rightarrow Y$ such that the diagram
\begin{equation}
	\begin{diagram}
		\node{\coprod_{c\in\C} F(c)}
			\arrow{e,t}{\coprod g_c}
			\arrow{s,l}{\subseteq}
		\node{Y}
			\arrow{s,r}{p}
	\\
		\node{X_n}
			\arrow{e,t}{f|_{X_n}}
			\arrow{ne,t}{g}
		\node{Z}
	\end{diagram}
\end{equation}
commutes. Denote $\mathcal{E}_0^\infty:=\lim\; \mathcal{E}^n_0$; note that this is the set of solutions of the lifting extension problem \ref{e:HLDiag}. For $g\in \mathcal{E}_0^n$ and $k\leq n\leq m$ let $\mathcal{E}_k^m(g)$ be the space of maps $h:X_m\rightarrow Y$ such that the diagram 
\begin{equation}
\begin{diagram}
		\node{X_k}
			\arrow{e,t}{g|_{X_k}}
			\arrow{s,l}{\subseteq}
		\node{Y}
			\arrow{s,r}{p}
		\\
		\node{X_{m}}
			\arrow{e,t}{f|_{X_m}}
			\arrow{ne,t}{h}
		\node{Z}
	\end{diagram}
\end{equation}
commutes.

\subsection*{Obstruction cochains and difference cochains}
For $g\in\mathcal{E}_0^n$ define \emph{the obstruction cochain} $\mathfrak{O}_{n}(g)\in C^{n+1}_n$ by
\begin{equation}
	\mathfrak{O}_n(g)(\sigma):= \mathfrak{o}_n(\Ad[g]\sigma,\Ad[f]\sigma) \in \pi_n(Fib(\sigma(0))),
\end{equation}
where $\mathfrak{o}_n(\Ad[g]\sigma,\Ad[f]\sigma)$ is the obstruction of the lifting problem
\begin{equation}\label{e:LiftingDiagM}
\begin{diagram}
		\node{\partial\Delta^{n+1}}
			\arrow{e,t}{\Ad[g]{\sigma}}
			\arrow{s,l}{\subseteq}
		\node{\map(F(\sigma(0)),Y)_{g_{\sigma(0)}}}
			\arrow{s,r}{\map(F(\sigma(0)),p)}
		\\
		\node{\Delta^{n+1}}
			\arrow{e,t}{\Ad[f]\sigma}
			\arrow{ne,..}
		\node{\map(F(\sigma(0)),Z)_{f_{\sigma(0)}}}
	\end{diagram}
\end{equation}

For $h,h'\in\mathcal{E}^{n+1}_n(g)$ define \emph{the difference cochain} $\mathfrak{D}_{n+1}(h',h)\in C_{n+1}^{n+1}$ by
\begin{equation}
	\mathfrak{D}_{n+1}(h',h)(\sigma)=\mathfrak{d}_{n+1}(\Ad[h']\sigma,\Ad[h]\sigma)\in \pi_{n+1}(\Fib(\sigma(0))).
\end{equation}

\begin{prp}\label{p:ObstrProp}
	Fix $g\in\mathcal{E}_0^n$, $n>0$.
	\begin{enumerate}
		\item{$\mathcal{E}^{n+1}_n(g)\neq \emptyset$ if and only if $\mathfrak{O}_n(g)=0$.}
		\item{$\mathfrak{O}_{n}(g)\in Z^{n+1}_n$.}
		\item{$\mathfrak{O}_n(g')=\mathfrak{O}_n(g)+\delta^n_n(\mathfrak{D}_{n}(g',g))$ for any $g'\in \mathcal{E}_{n-1}^{n}(g)$.}
		\item{For every $u\in C_{n+1}^{n+1}$ and every $h\in \mathcal{E}^{n+1}_n(g)$ there exists a map $h'\in \mathcal{E}^{n+1}_n(g)$ such that $\mathfrak{D}_{n+1}(h',h)=u$.} 
	\end{enumerate}
\end{prp}
\begin{proof}
	Existence of $h\in \mathcal{E}^{n+1}_n(g)$ is equivalent to existence of a family of maps $\{\Ad[h]\sigma\}\in \mathcal{L}(\Ad[g]\sigma,\Ad[f]\sigma)$ for $\sigma\in N(\C)_{n+1}$. Then (1) is a straightforward consequence of \ref{p:ObstrE}.(1). For similar reasons, (4) is a consequence of \ref{p:ObstrE}.(4). To prove (2), we check that for $\sigma\in N(\C)_{n+2}$ we have
	\begin{multline*}
		\delta_n^{n+1}(\mathfrak{O}_n(g))(\sigma)=\Fib(\sigma(0\rightarrow 1))^*(\mathfrak{O}_n(g)(d_0\sigma))+\sum_{i=1}^{n+2} (-1)^i \mathfrak{O}_n(g)(d_i\sigma)=\\
		\Fib(\sigma(0\rightarrow 1))^*\mathfrak{o}_n(\Ad[g]{d_0\sigma},\Ad[f]{d_0\sigma})+\sum_{i=1}^{n+2} (-1)^i \mathfrak{o}_n(\Ad[g]{d_i\sigma},\Ad[f]{d_i\sigma}) \buildrel{\text{by \ref{p:ObstrE}(2)}}\over{=}\\
		\mathfrak{o}_n(\Ad[g]{d_0\sigma}\circ F^*(\sigma(0\rightarrow 1)), \Ad[f]{d_0\sigma}\circ F^*(\sigma(0\rightarrow 1)))+ \sum_{i=1}^{n+2} (-1)^i \mathfrak{o}_n(\Ad[g]{d_i\sigma},\Ad[f]{d_i\sigma})\\
		\buildrel{\text{by \ref{e:AdComp}}}\over{=}\sum_{i=0}^{n+2} (-1)^i \mathfrak{o}_n(\Ad[g]\sigma\circ d_i, \Ad[f]\sigma\circ d_i)\buildrel{\text{by \ref{p:ObstrE}(3)}}\over{=}0.
	\end{multline*}
	Now let  $\sigma\in N(\C)_{n+1}$. As a consequence of \ref{e:AdComp}, for $i>0$ we have
	\[
		\mathfrak{d}_{n}(\Ad[g']{\sigma}\circ d_i,\Ad[\sigma]{g}\circ d_i)=\mathfrak{d}_{n}(\Ad[g']{d_i\sigma},\Ad[g]{d_i\sigma})=\mathfrak{D}_n(g',g)(d_i\sigma)
	\]
	and
	\begin{multline*}
		\mathfrak{d}_{n}(\Ad[g']\sigma\circ d_0,\Ad[g]\sigma\circ d_0)=\mathfrak{d}_{n}(F(\sigma(0\rightarrow 1))\circ \Ad[g']{d_0\sigma},F(\sigma(0\rightarrow 1))\circ \Ad[g]{d_0\sigma})\buildrel{\ref{p:ObstrF}.(2)}\over{=}\\
		\Fib(\sigma(0,1))_*\left( \mathfrak{d}_n(\Ad[g']{d_0\sigma},\Ad[g]{d_0\sigma})\right)=\Fib(\sigma(0,1))_*(\mathfrak{D}_n(g',g)(d_0\sigma)).
	\end{multline*}
	Finally,
	\begin{multline*}
		\mathfrak{O}_n(g')(\sigma)=\mathfrak{o}_n(\Ad[g']{\sigma},\Ad[f]\sigma)
			\buildrel{\ref{p:ObstrE}.(2)}\over{=}\\
		\mathfrak{o}_n(\Ad[g]\sigma, \Ad[f]\sigma)+\sum_{i=0}^{n+1}(-1)^i\mathfrak{d}_n(\Ad[g']\sigma\circ d_i, \Ad[g]{\sigma}\circ d_i)=\\
		\mathfrak{O}_n(g)(\sigma)+\Fib(\sigma(0\rightarrow 1))_*\left(\mathfrak{D}_n(g',g)(d_0\sigma)\right)
		+\sum_{i=1}^{n+1}(-1)^i\mathfrak{D}_n(g',g)(d_i\sigma)=\\
		=	\mathfrak{O}_n(g)(\sigma) + \delta_n^n(\mathfrak{D}_n(g',g)).	
	\end{multline*}
	This proves (3).
\end{proof}

\begin{lem}\label{p:OneStepExt}
	Fix $n>0$ and assume that $H^{n+1}(\mathcal{C};\pi_n(\Fib))=0$. For every $g\in \mathcal{E}_0^n$ there exists $h\in \mathcal{E}^{n+1}_{n-1}(g)$.
\end{lem}
\begin{proof}
	By \ref{p:ObstrProp}.(2) $\mathfrak{O}_n(g)\in Z_n^{n+1}$. Since $H^{n+1}_n=0$, there exists $u\in C^n_n$ such that $\delta_n^n(u)=\mathfrak{O}_n(g)$. Let $g'\in\mathcal{E}_{n-1}^{n}(g)$ be a map such that $\mathfrak{D}_n(g',g)=-u$; it exists by \ref{p:ObstrProp}.(4). By \ref{p:ObstrProp}.(3) we have
\[
	\mathfrak{O}_n(g')=\mathfrak{O}_n(g)+\delta_n^n(\mathfrak{D}_n(g',g))=\mathfrak{O}_n(g)+\delta_n^n(-u)=0.
\]	
	 The conclusion follows by \ref{p:ObstrProp}.(1).
\end{proof}

Finally, we are ready to prove the main theorem of this section:
\begin{proof}[Proof of {\ref{t:Obstr}}]
	We need to prove that the space $\mathcal{E}^\infty_0:=\lim\; \mathcal{E}^n_0$ is non-empty. Since $\{g_c\}_{c\in\Ob(\C)}\in\lim_{\C}\;[F,Y]$, the map $g_{0}=\coprod g_c:X_0\rightarrow Y$ extends to a map $X_1\rightarrow Y$, which can be adjusted (since $p$ is a fibration) to a map $g_1\in \mathcal{E}^1_0$. Using \ref{p:OneStepExt}  we inductively construct a sequence of maps $g_n$, $n\geq 2$, such that $g_n\in\mathcal{E}^n_{n-2}(g_{n-1})$. Put $g'_n:=g_{n+1}|_{X_n}$. Since $g'_n|_{X_{n-1}}=g'_{n-1}$, the sequence  $\{g'_n\}$ represents an element in $\mathcal{E}^\infty_0$.
\end{proof}

\section{Representations of $p$-stubborn subgroups of $U(n)$}

\label{s:RepresentationsOfStubbornSubgroups}

Let $m\in\Z$ be an integer such that $(m,n!)=1$. Recall (\cite{Su}, \cite{JMO}) that there exists a unique homotopy representation $\Psi^m:BU(n)\to BU(n)$ such that $\Psi^m|_{BT}$ is a map induced by the composition
\[
	T\ni x \mapsto x^m\in T\subseteq U(n).
\]
Let $\psi^m\in R(U(n))$ be the character of $\Psi^m$ and fix a prime integer $p$. In this Section we prove Proposition \ref{p:EpiSigma} which states that for every $p$-stubborn subgroup $P\subseteq U(n)$ and every $N_{U(n)}(P)/P$-orbit $X\subseteq \IR(P,\psi^m_P)$ the map $N_{U(n)}(P)/P\to\Sigma_X$ is surjective. 

Let us recall the classification of $p$-stubborn subgroups of unitary groups obtained by Oliver \cite{O}. Define $p\times p$-matrices 
\begin{equation}
	A=\begin{pmatrix}
	1 & 0 & 0 & \cdots & 0 \\	
	0 & \zeta & 0 & \cdots & 0 \\	
	0 & 0 & \zeta^2 & \cdots & 0 \\	
	\vdots & \vdots & \vdots & \ddots & \vdots \\
	0 & 0 & 0 & \dots & \zeta^{p-1} \\	
	\end{pmatrix}
	\;\;\;\;
	B=\begin{pmatrix}
	0 & 1 & 0 & \cdots & 0 \\	
	0 & 0 & 1 & \cdots & 0 \\	
	\vdots & \vdots & \vdots & \ddots & \vdots \\
	0 & 0 & 0 & \cdots & 1 \\	
	1 & 0 & 0 & \dots & 0 \\	
	\end{pmatrix}
\end{equation}
where $\zeta=e^{2\pi i/p}$. Let $I_k$ denote the $k\times k$ identity matrix. Define matrices $A^k_i,B^k_i\in U(p^k)$ for $i=0,\dots,k-1$ by
\begin{equation}
	A^k_i=I_{p^i}\otimes A\otimes I_{p^{k-i-1}},\;\;\;\;
	B^k_i=I_{p^i}\otimes B\otimes I_{p^{k-i-1}}.
\end{equation}
Finally, let
\begin{equation}
	\Gamma(k):=\langle S^1, A^k_0,\dots, A^k_{k-1},B^k_0,\dots, B^k_{k-1} \rangle \subseteq U(p^k).
\end{equation}
Note that $(A^k_i)^p=(B^k_i)^p=I$ and
\begin{equation}\label{e:GammaRel}
	A^k_i A^k_j =  A^k_j A^k_i,\quad 	B^k_i B^k_j =  B^k_j B^k_i, \quad
	B^k_i A^k_j=\begin{cases}
		\zeta\cdot A^k_j B^k_i & \text{for $i=j$}\\
		A^k_j B^k_i & \text{for $i\neq j$.}
	\end{cases}
\end{equation}
Furthermore, every element of $x\in \Gamma(k)$ can be written uniquely in the form
\begin{equation}\label{e:GammaPres}
	x=t\cdot A^k_{i_1}\dots A^k_{i_r} B^k_{j_1}\dots B^k_{j_s},
\end{equation}
where $t\in S^1$, and $0\leq i_1<\dots<i_r<k$, $0\leq j_1<\dots<j_s<k$.

Let $C_p$ denote the cyclic group of order $p$. We say that a subgroup of a unitary group is \emph{elementary $p$-stubborn} if it has the form
\[
	\Gamma(k;a_1,\dots,a_r):=\Gamma(k)\wr C_p^{a_1}\wr \dots \wr C_p^{a_r}\subseteq U(p^k)\wr C_p^{a_1}\wr \dots \wr C_p^{a_r} \subseteq U(p^{k+\sum a_i}).
\]
If $P=\Gamma(k;a_1,\dots,a_r)$ is an elementary $p$-stubborn subgroup, then \cite[Th. 6]{O}
\[N_{U(n)}(P)/P\simeq \Sympl_{2k}(\F_p)\times \GL_{a_1}(\F_p)\times\dots\times \GL_{a_r}(\F_p).
\]

\begin{prp}[{\cite[Th. 8]{O}}]\label{p:StubbornSubgroupOfUn}
	Every $p$-stubborn subgroup in $U(n)$ is conjugate to
	\[
		P:=P_1^{b_1}\times \dots \times P_j^{b_j}\subseteq U(p^{n_1})^{b_1}\times \dots\times U(p^{n_j})^{b_j}\subseteq U(n),
	\]
	where $\{P_i\subseteq U(p^{n_i})\}$ is a family of pairwise non-isomorphic elementary $p$-stubborn subgroups. Furthermore,
	\begin{equation}\label{e:WeylGroupOfGeneralStubborn}
		N_{U(n)}(P)/P\simeq \prod_{i=1}^j (N_{U(p^{n_i})}(P_i)/P_i)\wr \Sigma_{b_i}.\qed	
	\end{equation}
\end{prp}



\begin{prp}\label{p:IrrIsIrr}
	If $n=p^k$ then the representation $\psi^m_{\Gamma(k)}$ is irreducible.
\end{prp}
\begin{proof}
	If $k=0$, then $\psi^m_{\Gamma(k)}$ has dimension 1. Otherwise, $p$ does not divide $m$. Define an automorphism $\alpha:\Gamma(k)\to \Gamma(k)$ by formulas
	\begin{itemize}
		\item{$\alpha(x)=x^m$ for $x\in S^1\cong Z(U(p^k))$,}
		\item{$\alpha(A^k_i)=A^k_{mi\;\mathrm{mod}\;p}=(A^k_i)^m$,}
		\item{$\alpha(B^k_i)=B^k_i$.}
	\end{itemize}
	By checking conditions (\ref{e:GammaRel}) we prove that $\alpha$ is well-defined, and since $(p,m)=1$ it is surjective. Let 
	\[x=t\cdot A^k_{i_1}\dots A^k_{i_r} B^k_{j_1}\dots B^k_{j_s}\in\Gamma(k),\]
	be the presentation (\ref{e:GammaPres}) of an arbitrary element of $\Gamma(k)$. If $s=0$ then $x^m=\alpha(x)$; for $s>0$ both $x^m$ and $\alpha^m(x)$ have no non-zero elements on the diagonal and then $\tr(x^m)=\tr(\alpha(x))=0$. Thus the composition $\iota_{\Gamma(k)}\circ \alpha$ is conjugate to $\psi^m_{\Gamma(k)}$ since they have equal characters. Now the conclusion follows from irreducibility of $\iota_{\Gamma(k)}$, which is proven in the proof of \cite[Th. 6]{O}, and surjectivity of $\alpha$.
\end{proof}

\begin{prp}\label{p:IdentityIsIrreducible}
	For every elementary $p$-stubborn subgroup $P$ of $U(n)$ the representation $\psi^m_P:P\subseteq U(n)$ is irreducible.
\end{prp}
\begin{proof}
	Induction with respect to $j$, where $P=\Gamma(k;a_1,\dots,a_j)$. The case $j=0$ is proven in \ref{p:IrrIsIrr}. To show the general case, note that
	\[
		\res^{\Gamma(k;a_1,\dots,a_j)}_{\Gamma(k;a_1,\dots,a_{j-1})^{p^{a_j}}}\psi^m_{\Gamma(k;a_1,\dots,a_j)}\simeq \bigoplus_{m=1}^{p^{a_j}}
		\psi^m_{\Gamma(k;a_1,\dots,a_{j-1})}\circ pr_m	,	
	\]
	where $pr_m$ is the projection on the $m$-th factor. The quotient $\Gamma(k;a_1\dots a_j)/\Gamma(k;a_1,\dots,a_{j-1})^{p^{a_j}}\cong C_p^{a_j}$ acts transitively on the summands at the right-hand side. Hence the representation induced from any single summand is irreducible and by Frobenius reciprocity it is isomorphic to $\psi^m_{\Gamma(k;a_1,\dots,a_j)}$.
\end{proof}

\begin{proof}[Proof of {\ref{p:EpiSigma}}]
	By \ref{p:StubbornSubgroupOfUn} we can assume that $P=\prod_{i=1}^j P_i^{b_i}$, where $P_i\subseteq U(p^{r_i})$ are pairwise non-isomorphic elementary $p$-stubborn subgroups. For $i=1,\dots,j$ and $k=1,\dots,b_i$ define a unitary representation of $P$ 
	\[
		\gamma_{i,k}= \iota_{P_i} \circ pr_{i,k},
	\]
	where $pr_{i,k}:P\to P_i$ is the projection onto $k$-th summand of type $P_i$.	Obviously 	$\iota_P\simeq \bigoplus_{i=1}^{j}\bigoplus_{k=1}^{b_i} \gamma_{i,k}$. As a consequence of \ref{p:IdentityIsIrreducible}, all representations $\gamma_{i,k}$ are irreducible; therefore $\IR(P,\iota_P)=\{\gamma_{i,k}\}_{i=1,\dots,j}^{k=1,\dots,b_i}$. For every $\eta\in N_{U(n)}(P)/P$ we have $\eta^*\gamma_{i,k}=\gamma_{i,p_i(\eta)(k)}$, where 
	\[
	p_{i}:N_{U(n)}(P)/P\rightarrow (N_{U(n_i)}(P_i)/P)\wr \Sigma_{b_i}\rightarrow \Sigma_{b_i}\simeq \Sigma_{\{\gamma_{i,k}\}_{k=1}^{b_i}}
	\] 
	is given by the projection on the $i$-th summand of \ref{e:WeylGroupOfGeneralStubborn}. Then every $N_{U(n)}(P)/P$-orbit of $\IR(P,\iota_P)$ has the form ${\{\gamma_{i,k}\}_{k=1}^{b_i}}$ and the conclusion follows from surjectivity of homomorphisms $p_i$.
\end{proof}

\section{Calculation of functors $\Lambda^*$}

\label{s:FunctorsLambdaSigma}

Fix a prime $p$. Let $X$ be a finite set and let $\Sigma_X$ be the group of permutations of $X$. In this Section we prove that $\Lambda^i(\Sigma_X;\Z_p^\wedge[X])=0$ for $i\geq 2$. To achieve this, we use results of Aschbacher, Kessar and Oliver \cite{AKO}. For a finite group $H$ let $O_p(H)$ be the maximal normal $p$-subgroup of $H$. Note that a $p$-subgroup $P\subseteq H$ is $p$-radical if and only if $P=O_P(N_H(P))$.
\begin{df}
	Let $\Gamma$ be a finite group. \emph{A radical $p$-chain of length $k$} in $\Gamma$ is a sequence 
	\[
		P_0\subsetneq P_1\subsetneq P_2\subsetneq \dots\subsetneq P_k \subseteq \Gamma
	\]
	of distinct $p$-subgroups of $\Gamma$ such that
	\begin{itemize}
		\item{$P_0=O_p(\Gamma)$,}
		\item{$P_i=O_p(N_\Gamma(P_0)\cap N_\Gamma(P_1)\cap \dots\cap N_\Gamma(P_i))$,}
		\item{$P_k$ is a $p$-Sylow subgroup of $N_\Gamma(P_0)\cap N_\Gamma(P_1)\cap \dots\cap N_\Gamma(P_{k-1}))$.}
	\end{itemize}
\end{df}

\begin{lem}\label{l:AKO}
	Fix a finite group $\Gamma$, and a finitely generated $\Z_p^\wedge[\Gamma]$-module $M$. Assume, for some $k\geq 1$, that $\Lambda^k(\Gamma;M)\neq 0$. Then there is a radical $p$-chain
	\[
		1=P_0\subsetneq P_1\subsetneq P_2\subsetneq \dots\subsetneq P_k \subseteq \Gamma
	\]
	such that $M/pM$, regarded as an $\F_p[P_k]$-module, contains a copy of the free module $\F_p[P_k]$.
\end{lem}
\begin{proof}
	It is a special case of \cite[Lemma 5.27]{AKO}.
\end{proof}

\begin{prp}\label{p:RankOfRadical}
	Assume that $X$ is a finite set having $n$ elements and that $P\subseteq \Sigma_X$ is a $p$-radical subgroup. Then
	\[
		|P|\geq n-|X^P|.
	\]
\end{prp}
\begin{proof}
	Let $\mathfrak{S}$ be the set of all finite (possibly empty) sequences of positive integers. For all  $\mathbf{k}=(k_1,\dots,k_r)\in\mathfrak{S}$ define inductively subgroups $A(\mathbf{k})\subseteq \Sigma_{p^{|\mathbf{k}|}}$, $|\mathbf{k}|:=k_1+\dots+k_r$, by $A(\emptyset)=1\subseteq\Sigma_1$, and
\[
	A(k_1,\dots,k_r)=A(k_1,\dots,k_{r-1})\wr C_p^{k_r}\subseteq \Sigma_{p^{k_1+\dots+k_{r-1}}}\wr C_{p}^{k_r}\subseteq \Sigma_{p^{|\mathbf{k}|}}.
\]
	Following \cite{AF}, every $p$-radical subgroup of $\Sigma_X$ has the form 
	\[
	P=\prod_{\mathbf{k}\in\mathfrak{S}}A(\mathbf{k})^{m(\mathbf{k})}\subseteq \prod_{\mathbf{k}\in\mathfrak{S}} \left(\Sigma_{p^{|\mathbf{k}|}}\right)^{m(\mathbf{k})}\subseteq \Sigma_n\simeq \Sigma_X,
	\]
where $n=\sum m(\mathbf{k})p^{|\mathbf{k}|}$. Notice that $|A(\mathbf{k})|\geq p^{|\mathbf{k}|}$ and that $A(\emptyset)$ is a trivial group. Thus
\[
	|P|=\prod_{\mathbf{k}\in\mathfrak{S}\setminus\{\emptyset\}} |A(\mathbf{k})|^{m(\mathbf{k})}\geq \prod_{\mathbf{k}\in\mathfrak{S}\setminus\{\emptyset\}} p^{|\mathbf{k}|\cdot m(\mathbf{k})}\geq \prod_{\mathbf{k}\in\mathfrak{S}\setminus\{\emptyset\}} m(\mathbf{k})\cdot p^{|\mathbf{k}|}=n-m(\emptyset).
\]
Since $|X^P|=m(\emptyset)$, the conclusion follows.
\end{proof}

\begin{prp}\label{p:ChainSize}
	Let $X$ be a finite set having $n$ elements. Assume that 
	\[
		1=P_0\subsetneq P_1\subsetneq P_2\subsetneq \dots\subsetneq P_k \subseteq \Sigma_X
	\]
	is a radical $p$-chain of length $k\geq 2$. Then either $|P_k|>n$, or $p=2$, $n=5$ and this chain is conjugate to 
	\[
		1=P_0\subsetneq P_1 =\langle a \rangle \subsetneq P_2 =\langle a,b \rangle \subseteq\Sigma_X,
	\]
	where $a, b \in\Sigma_X$ are disjoint cycles of length 2.
\end{prp}
\begin{proof}
	We can assume that $n\geq p^2$ since otherwise all $p$-subgroups of $\Sigma_X$ have orders less or equal $p$. Denote $X_i=X^{P_i}$, $Y_i=X\setminus X_i$; obviously $X_0=X$. As a consequence of the definition of radical $p$-chain (cf.\ \cite{AKO}), $P_i/P_{i-1}$ is a $p$-radical subgroup of the group
	\begin{equation*}
		\left(\bigcup_{j=0}^{i-1 }N_{\Sigma_X}(P_j)\right)/P_{i-1}= \left(\bigcup_{j=0}^{i-1 }N_{\Sigma_{Y_{i-1}}}(P_j)\right)/P_{i-1} \times \Sigma_{X_{i-1}}.
	\end{equation*}
	Hence, by \cite[Th. 1.6.(ii)]{JMO} it is a product $H\times H'$ of $p$-radical subgroups
	\[
		H\subseteq  \left(\bigcup_{j=0}^{i-1 }N_{\Sigma_{Y_{i-1}}}(P_j)\right)/P_{i-1} \quad\text{and}\quad H'\subseteq \Sigma_{X_{i-1}}.
	\]
	Clearly $X_i=X_{i-1}^{H'}$ and by \ref{p:RankOfRadical} $|P_i:P_{i-1}|\geq |H'|\geq |X_{i-1}\setminus X_{i}|$. Thus 
	\[
		|P_k|=\prod_{i=1}^k|P_i:P_{i-1}|\geq \prod_{i=1}^k\max\{p,|X_{i-1}\setminus X_{i}|\}.
	\]
	Since $\sum_{i=1}^k |X_{i-1}\setminus X_i|=|X_0\setminus X_k|> n-p$, we obtain $P_k>n$ unless $p=2$ and $n\in\{4,5\}$. But the case $n=4$ is excluded since $\Sigma_4$ is not $2$-reduced, and the only possibility for $n=5$ is stated above.
\end{proof}

\begin{proof}[Proof of {\ref{p:LambdaSigma}}]
	Without loss of generality we can assume that $X=\{x_1,\dots,x_n\}$. Assume that \hbox{$\Lambda^i(\Sigma_X; \Z_p^\wedge[X])\neq 0$} for some $i\geq 2$. From \ref{l:AKO} follows that there exists a radical $p$-chain $P_0\subsetneq \dots\subsetneq P_i$ in $\Sigma_n$ which admits a $P_i$-equivariant monomorphism $\varphi:\F_p[P_i]\to\F_p\{x_1,\dots,x_n\}$. By \ref{p:ChainSize} and dimensional reasons, $(i,p,n)=(2,2,5)$ and $P_2\subseteq\Sigma_5$ is a subgroup generated by two disjoint cycles of length 2. We may assume that $P_2=\langle a,b\rangle$, where $a$ (resp.\ $b$) is the transposition which exchanges $x_1$ and $x_2$ (resp.\ $x_3$ and $x_4$). Let $c_1,\dots,c_5\in\F_2$ be elements such that $\varphi(e)=\sum_{i=1}^5 c_ix_i$, where $e\in P_2$ is the identity element. We have 
	\begin{multline*}
		\varphi(e-a-b+(ab))=(c_1x_1+c_2x_2+c_3x_3+c_4x_4+c_5x_5)-(c_2x_1+c_1x_2+c_3x_3+c_4x_4+c_5x_5)\\
		-(c_1x_1+c_2x_2+c_4x_3+c_3x_4+c_5x_5) +(c_2x_1+c_1x_2+c_4x_3+c_3x_4+c_5x_5)=0,
	\end{multline*}
	which contradicts injectivity of $\varphi$. 
\end{proof}

\section{Homotopy actions}

\label{s:HomotopyActions}

\def\Rep{\mathbf{Rep}}
\def\HSp{\mathbf{HSp}}
Let $\Rep$ be the category of compact Lie groups and their representations, i.e.\ 
\[
\Mor_{\Rep}(G,H):=\Hom(G,H)/\Inn(H).
\]
Fix a finite group $N$; it will be regarded as the category with a single object having $N$ as the group of automorphisms.

\begin{df}\label{d:hNActions}
	\emph{A left (resp.\ right) homotopy $N$-action} (or \emph{an h-$N$-action} in short) on a compact Lie group $\Gamma$ is a functor $N\to \Rep$ (resp. $N^{op}\to\Rep$) whose value on the single object of $N$ is $\Gamma$. Similarly, \emph{a left (resp. right) h-$N$-action} on a space $X$ is a functor $N\to \HSp$ (resp. $N^{op}\to\HSp$) with value $X$. A group $\Gamma$ (resp. a space $X$) equipped with a left (resp. right) h-$N$-action will be called a left h-$N$-group (resp. right h-$N$-group, left h-$N$-space, right h-$N$-space). 	Let $N\text{-}\Rep$ (resp.\ $N^{op}$-$\Rep$, $N$-$\HSp$, $N^{op}$-$\HSp$) denote the category of left h-$N$-groups (resp.\ right h-$N$-groups, left h-$N$-spaces, right h-$N$-spaces), where morphisms are natural transformations of functors. Either a homomorphism or a map is called \emph{h-$N$-equivariant} (resp.\ h-$N$-equivalence) if it represents a morphism (resp.\ an isomorphism) in the suitable category, i.e.\ if it preserves an h-$N$-action.
\end{df}

Note that a map is an h-$N$-equivalence if and only if it is  both h-$N$-equivariant and a homotopy equivalence. The classifying space functor maps conjugate homomorphisms into homotopic maps. Thus, it defines functors $B:N\text{-}\Rep\to N\text{-}\HSp$ and $B:N^{op}\text{-}\Rep\to N^{op}\text{-}\HSp$.

\begin{df}
	Let $K$ be a left h-$N$-group and let $H$ be a compact Lie group. We say that a homomorphism $\alpha:K\to H$ is \emph{h-$N$-invariant} if for every $\eta\in N$ the representations $[\alpha]$ and $[\alpha]\circ \eta$ are equal (as morphisms in $\Rep$).
\end{df}

\subsection*{Functoriality of centralizers}

Let $H$ be a compact Lie group, $P$ a $p$-toral left \hbox{h-$N$-group}, and $\alpha:P\to H$ an h-$N$-invariant homomorphism. For a homomorphism $\gamma:P\to P$ representing the image of $\eta\in N$ in $\Out(N)$ there exists $b_\gamma\in H$ such that the diagram
\begin{equation}\label{e:CentrFunc}
	\begin{diagram}
		\node{P}
			\arrow{e,t}{\alpha}
			\arrow{s,l}{\gamma}
		\node{H}
			\arrow{s,r}{h\mapsto b_{\gamma}^{-1}hb_{\gamma}}
	\\
		\node{P}
			\arrow{e,t}{\alpha}
		\node{H}
	\end{diagram}
\end{equation}
is strictly commutative. The formula
\begin{equation}\label{e:CInducedMap}
	C_H(\alpha(P))\ni h \mapsto b_\gamma h b_\gamma^{-1}\in C_H(\alpha(P))
\end{equation}
defines a homomorphism which does not depend, up to conjugacy, on the choice of representative $\gamma$. To prove independence (up to conjugacy again) on the choice of element $b_\gamma$, fix another $b'_\gamma\in H$ such that $(b'_\gamma)^{-1}\alpha(g)b'_\gamma=\alpha(\gamma(g))$ for all $g\in P$. It is enough to prove that $b'_\gamma b_\gamma^{-1}\in C_H(\alpha(P))$, which follows from the following calculation:
\[
	(b'_\gamma b_\gamma^{-1})^{-1}\alpha(g)(b'_\gamma b_\gamma^{-1})= b_\gamma ((b'_\gamma)^{-1}\alpha(g) b'_\gamma) b_\gamma^{-1}=b_\gamma ((b_\gamma)^{-1}\alpha(g) b_\gamma) b_\gamma^{-1}=\alpha(g).
\]
Thus, the formula (\ref{e:CInducedMap}) defines a right h-$N$-action on $C_H(\alpha(P))$.

\begin{prp}\label{p:DZNFunc}
	Let $P$ be a $p$-toral left h-$N$-group, H a compact Lie group, and $\alpha:P\to H$ an h-$N$-invariant homomorphism. The map (cf.\ \ref{t:DZN})
	\[
		(ad_\alpha)_p^\wedge: BC_{H}(\alpha(P))_p^\wedge\to (\map(BP, BH)_{B\alpha})_p^\wedge
	\]
	is an equivalence of right h-$N$-spaces.
\end{prp}
\begin{proof}
	We have to prove that for every $\eta\in N$ and $\gamma:P\to P$ representing the image of $\eta\in N$ in $\Out(N)$ the diagram
	\[
		\begin{diagram}
			\node{BC_H(\alpha(P))}
				\arrow{e,t}{ad_{\alpha}}
			\node{\map(BP,BH)_{B\alpha}}
		\\
			\node{BC_H(\alpha(P))}
				\arrow{e,t}{ad_{\alpha}}
				\arrow{n,l}{\eta}
			\node{\map(BP,BH)_{B\alpha}}
				\arrow{n,r}{\map(B\gamma,BH)_{B\alpha}}
		\end{diagram}
	\]
	commutes up to homotopy. After passing to adjoint maps, this reduces to checking that maps
	\[
		B\varphi, B\psi: B(C_H(\alpha(P))\times P)\rightarrow BH
	\]
	are homotopic, where 
	\begin{align*}
		\varphi:C_H(\alpha(P))\times P \ni & (g,g') \mapsto b_\gamma gb_\gamma^{-1}\cdot \alpha(g')\in H\\
		\psi     :C_H(\alpha(P))\times P \ni & (g,g') \mapsto g\cdot \alpha(\gamma(g'))\in H,
	\end{align*}
	and $b_{\gamma}\in H$ is an element such that $b_{\gamma}^{-1}\cdot \alpha(g) b_{\gamma}=\alpha(\gamma(g))$ for all $g\in P$. Since $\varphi$ and $\psi$ are conjugate, they induce homotopic maps. Thus $ad_{\alpha}$ preserves h-$N$-action, and by \ref{t:DZN} $(ad_\alpha)_p^\wedge$ it is a homotopy equivalence. Then it is an equivalence of right h-$N$-spaces.
\end{proof}

\subsection*{Centralizers in unitary groups}
At this point we restrict to the case $H=U(d)$. As before, $P$ is a $p$-toral left h-$N$-group and $\alpha:P\to U(d)$ is an h-$N$-invariant homomorphism.

Recall that $\IR(P)$ denotes the set of isomorphism classes of unitary representations of $P$, and for $\varrho\in\IR(P)$ and a representation $\xi$ of $P$ $c_\xi^\varrho$ denotes the number of summands isomorphic to $\varrho$ in a decomposition of $\xi$ into a sum of irreducible representations. Define a group
\begin{equation}\label{e:DefJ}
	J_\alpha=\prod_{\varrho\in\IR(P)}U(c_\alpha^\varrho).
\end{equation}
The formula
\begin{equation}\label{e:ActionOnJ}
	\eta: J_\alpha\ni (f_\varrho)_{\varrho\in \IR(P)}\mapsto (f_{\eta^*\varrho})_{\varrho\in \IR(P)}\in J_\alpha
\end{equation}
for $\eta\in N$, defines a (strict) right $N$-action on $J_\alpha$. In particular, $J_\alpha$ is a right h-$N$-group.

The homomorphism $\alpha$ defines a linear $P$-action on $\mathbb{C}^d$. Let
\begin{equation}\label{e:MonDec}
	\mathbb{C}^d\simeq \bigoplus_{\varrho\in\IR(P)}W_\varrho
\end{equation}
be the decomposition of $\mathbb{C}^d$ onto the sum of isotypical representations of $P$. An element $\phi\in U(d)$ centralizes $\alpha(P)$ if and only if for every $\varrho\in\IR(P)$ the restriction $\varphi|_{W_\varrho}$ is a $P$-equivariant automorphism of $W_\varrho$. Then the homomorphism
\[
	\prod_{\varrho\in\IR(P)} U_P(W_\varrho)\to C_{U(d)}(\alpha(P))
\]
is an isomorphism, where $U_P(W_\varrho)$ is the group of $P$-equivariant unitary automorphisms of $W_\varrho$.
For any $\varrho\in\IR(P)$ let $V_\varrho$ be a $P$-vector space which represents $\varrho$ and fix $P$-isomorphisms $i_\varrho:V_\varrho\otimes \mathbb{C}^{c_\alpha^\varrho}\to W_\varrho$. The homomorphism
\[
	k_\alpha^\varrho:U(c_\alpha^\varrho)\ni f\mapsto i_\varrho\circ(id_{V_\varrho}\otimes f)\circ i_\varrho^{-1}\in U_P(W_\varrho)
\]
is an isomorphism by Schur's Lemma. Finally, define an isomorphism
\begin{equation}
	k_\alpha:J_\alpha=\prod_{\varrho\in \IR(P)}U(c_{\alpha}^\varrho) \xrightarrow{\bigoplus k_\alpha^\varrho} \prod_{\varrho\in \IR(P)}U_P(W_\varrho)\xrightarrow{\cong} C_{U(d)}(\alpha(P)).
\end{equation}

\begin{prp}\label{p:JIso}
	The homomorphism $k_\alpha:J_\alpha\to C_{U(d)}(\alpha(P))$	is an equivalence of right h-$N$-groups.
\end{prp}
\begin{proof}
	Let $\gamma:P\to P$ be an isomorphism which represents the action of $\eta\in N$ on $P$ and let $b_\gamma\in U(d)$ be an element such that $b_\gamma^{-1}\alpha(g)b_\gamma=\alpha(\gamma(g))$ for $g\in P$ (cf.\ \ref{e:CentrFunc}). By Schur's lemma $b_\gamma$ maps $W_{\eta^*\varrho}$ isomorphically onto $W_{\varrho}$ for every $\varrho\in\IR(P)$. Then there is a presentation
\[b_\gamma=\bigoplus_{\varrho\in\IR(P)}b_\gamma^{\varrho}\]
where $b_\gamma^\varrho:W_{\eta^*\varrho}\rightarrow W_{\varrho}$ are isomorphisms. The diagram 
\[
	\begin{diagram}
		\node{U(c_\alpha^\varrho)}
			\arrow{e,t}{k_\alpha^\varrho}
		\node{U_P(W_{\varrho})}
	\\
		\node{U(c_\alpha^{\eta^*\varrho})}
			\arrow{e,t}{k_\alpha^{\eta^*\varrho}}
			\arrow{n,l}{id}
		\node{U_P(W_{\eta^*\varrho})}	
			\arrow{n,r}{h\mapsto b_\gamma^\varrho h (b_\gamma^\varrho)^{-1}}
	\end{diagram}
\]
commutes up to conjugacy since both compositions are isomorphisms between unitary groups which preserve subgroups of homotheties. After taking the product over all $\varrho\in\IR(P)$ we obtain the diagram
\[
	\begin{diagram}	
		\node{J_\alpha}
			\arrow{e,t,=}{=}
		\node{\prod_{\varrho\in\IR(P)} U(c_{\alpha}^\varrho)}
			\arrow{e,t}{k_\alpha}
		\node{\prod_{\varrho\in\IR(P)}U_P(W_\varrho)}
			\arrow{e,t}{\cong}
		\node{C_{U(d)}(\alpha(P))}		
	\\
		\node{J_\alpha}
			\arrow{e,t}{=}
		\node{\prod_{\varrho\in\IR(P)} U(c_{\alpha}^\varrho)}
			\arrow{n,l}{(f_\varrho)\mapsto (f_{\eta^*\varrho})}
			\arrow{e,t}{k_\alpha}
		\node{\prod_{\varrho\in\IR(P)}U_P(W_\varrho)}
			\arrow{e,t}{\cong}
			\arrow{n,l}{\prod (h\mapsto b_\gamma^\varrho h (b_\gamma^\varrho)^{-1})}
		\node{C_{U(d)}(\alpha(P))}		
			\arrow{n,l}{h\mapsto b_\gamma h b_\gamma^{-1}}
	\end{diagram}
\]	
which also commutes up to conjugacy. The conclusion follows.
\end{proof}

Immediately from \ref{p:DZNFunc} and \ref{p:JIso} we obtain the following
\begin{cor}\label{c:JMapIso}
	The composition
	\[
		(BJ_\alpha)_p^\wedge\xrightarrow{B(k_\alpha)_p^\wedge} BC_{U(d)}(\alpha(P))_p^\wedge\xrightarrow{(ad_\alpha)_p^\wedge} (\map(BP,BU(d))_{B\alpha})_p^\wedge
	\]
	is an equivalence of right h-$N$-spaces.
\end{cor}

\subsection*{Centralizers in products of unitary groups}
Let $\alpha:P\to U(d)$, $\beta:P\to U(d')$ be h-$N$-invariant homomorphisms. The inclusions
\[
	U(c^\varrho_\alpha)\times U(c^\varrho_\beta)\xrightarrow{\oplus} U(c^\varrho_\alpha+c^\varrho_\beta)=U(c^\varrho_{\alpha\oplus \beta})
\]
induce an $N$-equivariant inclusion $J_\alpha\times J_{\beta}\subseteq J_{\alpha\oplus\beta}$.

\begin{prp}\label{p:JProd}
	The diagram 
	\[
		\begin{diagram}	
			\node{J_\alpha\times J_\beta}
				\arrow{e,t}{k_\alpha\times k_\beta}
				\arrow{s,l}{\subseteq}
			\node{C_{U(d)}(\alpha(P))\times C_{U(d')}(\beta(P))}
				\arrow{s,r}{\oplus}
		\\
			\node{J_{\alpha\oplus\beta}}
				\arrow{e,t}{k_{\alpha\oplus\beta}}
			\node{C_{U(d+d')}((\alpha\oplus\beta)(P))}
		\end{diagram}
	\]
	is a commutative diagram in $N^{op}$-$\Rep$.
\end{prp}
\begin{proof}
	All homomorphisms in the diagram preserve right h-$N$-action (for horizontal ones this follows from \ref{p:JIso}). Let $\mathbb{C}^{d'}=\bigoplus_{\varrho} W'_\varrho$ be the decomposition of the representation space of $\beta$ into monotypical summands (cf.\ \ref{e:MonDec}). Let $k_\alpha^\varrho:U(c_\alpha^\varrho)\to U_P(W_\varrho)$, $k_\beta^\varrho:U(c_\beta^\varrho)\to U_P(W'_\varrho)$ be isomorphisms defining $k_\alpha$ and $k_\beta$. One can choose, for every $\varrho\in\IR(P)$, an isomorphism $k_{\alpha\oplus\beta}^\varrho:U(c_\alpha^\varrho+c_\beta^\varrho)\to U_P(W_\varrho\oplus W'_\varrho)$ such that $k_{\alpha\oplus\beta}^\varrho|_{U(c_\alpha^\varrho)\times U(c_\beta^\varrho)}=k_\alpha^\varrho\times k_\beta^\varrho$. For such choices commutativity of the diagram follows from definitions.
\end{proof}

For an h-$N$-invariant homomorphism $\xi:P\to U(k)$ denote 
\begin{align}\label{e:DefJJ}
	\mathfrak{J}_\xi&:= (BJ_\xi)_p^\wedge\\
	\mathfrak{C}_\xi&:=BC_{U(k)}(\xi(P))_p^\wedge\\
	\mathfrak{M}_\xi&:=(\map(BP,BU(k))_{B\xi})_p^\wedge.
\end{align}

\begin{prp}\label{p:JMDiag}
	The diagram
	\[
	\begin{diagram}
		\node{\mathfrak{J}_\alpha{\times}\mathfrak{J}_\beta}
			\arrow[3]{e,t}{B(k_\alpha\times k_\beta)_p^\wedge}
			\arrow{s,r}{B(\subseteq)_p^\wedge}
	\node{}		\node{}
		\node{\mathfrak{C}_\alpha\times\mathfrak{C}_\beta}
			\arrow[3]{e,t}{(ad_\alpha\times ad_\beta)_p^\wedge}
			\arrow{s,r}{\oplus}
		\node{}\node{}
		\node{\mathfrak{M}_\alpha{\times}\mathfrak{M}_\beta}
			\arrow{s,r}{\oplus}
	\\
		\node{\mathfrak{J}_{\alpha\oplus\beta}}
			\arrow[3]{e,t}{(Bk_{\alpha\oplus\beta})_p^\wedge}
			\node{}\node{}
		\node{\mathfrak{C}_{\alpha\oplus\beta}}
			\arrow[3]{e,t}{(ad_{\alpha\oplus\beta})_p^\wedge}
		\node{}\node{}
		\node{\mathfrak{M}_{\alpha\oplus\beta}}
	\end{diagram}
	\]
	is a commutative diagram in $N^{op}$-$\HSp$, and the horizontal maps are h-$N$-equivalences.
\end{prp}
\begin{proof}
	It is a consequence of \ref{p:DZNFunc}, \ref{p:JIso}, \ref{p:JProd}, and the naturality of the maps $ad$ with respect to the target space.
\end{proof}

\subsection*{$N$-equivariant fibrations}

\begin{df}
	We say that $N$-equivariant fibrations $p:E\to B$ and $p':E'\to B'$ are h-$N$-equivalent if there exists a strictly commutative diagram
	\[
		\begin{diagram}
			\node{E}
				\arrow{e,t}{g}
				\arrow{s,l}{p}
			\node{E'}
				\arrow{s,r}{p'}
		\\
			\node{B}
				\arrow{e,t}{f}
			\node{B'}
		\end{diagram}
	\]
	such that $f$ and $g$ are h-$N$-equivalences.
\end{df}

Define a space
\begin{equation}
	\mathfrak{J}_\alpha\widetilde{\times}\mathfrak{J}_\beta:=(EJ_{\alpha\oplus\beta}/(J_\alpha\times J_{\beta}))_p^\wedge
\end{equation}
and an $N$-equivariant fibration
\begin{equation}
	j_{\alpha,\beta}:\mathfrak{J}_\alpha\widetilde{\times}\mathfrak{J}_\beta=(EJ_{\alpha\oplus\beta}/(J_\alpha\times J_{\beta}))_p^\wedge\to (EJ_{\alpha\oplus\beta}/J_{\alpha\oplus \beta})_p^\wedge=\mathfrak{J}_{\alpha\oplus\beta}.
\end{equation}
Let $\mathfrak{G}_{\alpha,\beta}$ be a fiber of $j_{\alpha,\beta}$. Immediately from definition follows that there are h-$N$-equivalences
\begin{equation}\label{e:Grass}
		\mathfrak{G}_{\alpha,\beta}\simeq (J_{\alpha\oplus\beta}/(J_\alpha\times J_\beta))_p^\wedge \cong
		\prod_{\varrho\in\IR(P)}\left( U(c_\alpha^\varrho+c_\beta^\varrho)/U(c_\alpha^\varrho)\times U(c_\beta^\varrho) \right)_p^\wedge,
\end{equation}
where $N$ acts on the product of $p$-completed unitary Grassmannians by permuting factors, as in \ref{e:ActionOnJ}. Note that this definition is consistent with (\ref{e:DefFrakG}).

Assume that $P$ is a $p$-stubborn subgroup of $G$ and that $N=N_G(P)/P=\Aut_{\mathcal{R}_p(G)}(G/P)$. Let
\begin{equation}
	\mathfrak{M}^G_\alpha\widetilde{\times}\mathfrak{M}^G_\beta:=\map(EG\times_G G/P,(EU(d+d')/(U(d)\times U(d')))_p^\wedge)_{B(\alpha\oplus\beta)_p^\wedge},
\end{equation}
and for an h-$N$-invariant homomorphism $\xi:P\to U(k)$ denote
\begin{equation}
	\mathfrak{M}_\xi^G:=\map(EG\times_G G/P, BU(k)_p^\wedge)_{B\xi}.
\end{equation}
The inclusion $U(d)\times U(d')\subseteq U(d+d')$ induces an $N$-equivariant fibration
\begin{equation}\label{e:FibrM}
	m^G_{\alpha,\beta}:\mathfrak{M}^G_\alpha\widetilde{\times}\mathfrak{M}^G_\beta\to \mathfrak{M}^G_{\alpha\oplus\beta},
\end{equation}
where the right action of $N$ comes from the left action on $EG\times_G G/P$.

\begin{rem}
	In the situation of (\ref{e:LiftingGStrict}), for any $p$-stubborn subgroup $P\subseteq G$ the space $\Fib_{\mu,\nu}^p(G/P)$ is the fiber of fibration $m^G_{\mu_P,\nu_P}$.
\end{rem}

\begin{prp}\label{p:EquivFibr}
	The fibrations $j_{\alpha,\beta}$ and $m_{\alpha,\beta}^G$ are h-$N$-equivalent.
\end{prp}
\begin{proof}	
	There is a commutative diagram in $N^{op}\text{-}\HSp$ 
	\[
	\begin{diagram}
		\node{\mathfrak{J}_\alpha\widetilde{\times}\mathfrak{J}_\beta}
			\arrow{se,b}{j_{\alpha,\beta}}
		\node{\mathfrak{J}_\alpha{\times}\mathfrak{J}_\beta}
			\arrow{w,t}{\simeq}
			\arrow{e,t}{\simeq}
			\arrow{s,r}{B(\subseteq)_p^\wedge}
		\node{\mathfrak{M}_\alpha{\times}\mathfrak{M}_\beta}
			\arrow{s,r}{\oplus}
		\node{\mathfrak{M}^G_\alpha\widetilde{\times}\mathfrak{M}^G_\beta}
			\arrow{s,r}{m^G_{\alpha,\beta}}
			\arrow{w,t}{\simeq}
	\\
		\node{}
		\node{\mathfrak{J}_{\alpha\oplus\beta}}
			\arrow{e,t}{\simeq}
		\node{{\mathfrak{M}_{\alpha\oplus\beta}}}
		\node{{\mathfrak{M}^G_{\alpha\oplus\beta}}}
			\arrow{w,t}{\simeq}
	\end{diagram}
	\]
	where the middle square is the diagram \ref{p:JMDiag}, and the remaining equivalences are induced by h-$N$-equivalences $EG\times_G G/P\simeq BP$ and $EJ_{\alpha\oplus\beta}\simeq EJ_{\alpha}\times EJ_\beta$. Since $m^G_{\alpha,\beta}$ is a fibration, there exists a strictly commutative diagram
	\[	
		\begin{diagram}
			\node{\mathfrak{J}_\alpha\widetilde{\times}\mathfrak{J}_\beta}
				\arrow{e}
				\arrow{s,l}{j_{\alpha,\beta}}
			\node{\mathfrak{M}^G_\alpha\widetilde{\times}\mathfrak{M}^G_\beta}
				\arrow{s,r}{m_{\alpha,\beta}}
		\\
			\node{\mathfrak{J}_{\alpha\oplus\beta}}
				\arrow{e}
			\node{\mathfrak{M}^G_{\alpha\oplus\beta}}
		\end{diagram}
	\]
	such that the horizontal maps are h-$N$-equivalences.
\end{proof}

\section{Fibers of equivariant fibrations}

\label{s:Fibers}

If $p:E\to B$ is a fibration, then every path $\omega:[0,1]\to B$ connecting points $b_0,b_1\in B$ induces a weak homotopy equivalence $\omega_*:p^{-1}(b_0)\to p^{-1}(b_1)$ between fibers over its endpoints.  If $B$ is simply connected, then this map does not depend (up to homotopy) on the choice of a path; in such case we will denote $t_{b_0,b_1}:=\omega_*$. For a commutative diagram
\begin{equation} \label{e:NEqDiag}
		\begin{diagram}
			\node{E}
				\arrow{e,t}{g}
				\arrow{s,l}{p}
			\node{E'}
				\arrow{s,r}{p'}
		\\
			\node{B}
				\arrow{e,t}{f}
			\node{B'}
		\end{diagram}
\end{equation}
where $p, p'$ are fibrations and $B'$ is simply connected, we define an induced map between fibers  $F=p^{-1}(b_0)$ and $F'=(p')^{-1}(b'_0)$ induced by this diagram as a composition
\[ 
	T_f^g:F=p^{-1}(b_0)\xrightarrow{g|_{F}} (p')^{-1}(f(b_0))\xrightarrow{t_{f(b_0),b_0'}} (p')^{-1}(b'_0)=F'.
\]
Clearly it is well-defined up to homotopy.

Let $N$ be a group, $E$ and $B$ $N$-spaces and $p:E\to B$ an $N$-equivariant fibration over simply connected basis. For every fiber $F=p^{-1}(b_0)$, $b_0\in B$ the formula
\[
	N\ni \eta \mapsto T_{\eta:B\to B}^{\eta:E\to E} \in [F,F]
\]
defines an h-$N$-action on $F$. If the spaces in the diagram (\ref{e:NEqDiag}) are $N$-spaces and the maps are $N$-equivariant, then the induced map between fibers $T_f^g:F\to F'$ is h-$N$-equivariant. We will prove that this generalizes, under certain assumptions, to transformations of fibrations which are only h-$N$-equivariant. 

\begin{lem}\label{l:FibAct}
	Consider a diagram of spaces
\[
		\begin{diagram}
			\node{E}
				\arrow{e,tb,dd}{g_1}{g_2}
				\arrow{s,l}{p}
			\node{E'}
				\arrow{s,r}{p'}
		\\
			\node{B}
				\arrow{e,t}{f}
			\node{B'}
		\end{diagram}
\]	
	Assume that 
\begin{itemize}
\item{$fp=p'g_1=p'g_2=:q$,}
\item{the maps $g_1$ and $g_2$ are homotopic,}
\item{$p:E\to B$ and $p':E'\to B'$ are simple fibrations with fibers $F=p^{-1}(b_0)$ and $F'=(p')^{-1}(f(b_0))$ respectively, where $b_0\in B$,}
\item{the homomorphism 
	\[
		\pi_1(\map(E,E')_{g_1})\xrightarrow{\map(E,p')_*} \pi_1(\map(E,B')_q) 
	\]
	is surjective.}
\end{itemize}	
	Then the maps $g_i|_F:F\to F'$ are homotopic for $i=1,2$.
\end{lem}
\begin{proof}
	The map $\map(E,E')_{g_1}\to \map(E,B')_q$ is a fibration with the fiber
	\[
		M=\{g\in\map(E,E')_{g_1}:\; p'g=q\}.
	\]
	By the last assumption $M$ is path-connected. Furthermore, the image of the restriction $M\to \map(F,E')$ is contained in $\map(F,F')$. Hence, $g_1,g_2\in M$ restrict to homotopic maps $g_1|_F,g_2|_F\in \map(F,F')$.
\end{proof}

\begin{prp}\label{p:hNActCrit}
	Consider a commutative diagram of simply connected $N$-spaces
	\[
		\begin{diagram}
			\node{E}
				\arrow{e,t}{g}
				\arrow{s,l}{p}
			\node{E'}
				\arrow{s,r}{p'}
		\\
			\node{B}
				\arrow{e,t}{f}
			\node{B'}
		\end{diagram}
	\]
	Assume that the maps $p$ and $p'$ are $N$-equivariant fibrations and that the maps $f$ and $g$ are h-$N$-equivariant. Let $F=p^{-1}(b_0)$ and $F'=(p')^{-1}(b_0')$ be fibers of $p$ and $p'$ respectively; assume that they are also simply connected. If the homomorphism
	\[
		p'_*:\pi_1\map(E, E')_g\to \pi_1\map(E,B')_{p'g}
	\]
	is surjective, then the induced map $T_f^g:F\to F'$ between fibers is h-$N$-equivariant.
\end{prp}
\begin{proof}	
	Fix $\eta\in N$. Denote $f_1=B\xrightarrow{\eta}B\xrightarrow{f}B'$, $g_1=E\xrightarrow{\eta}E\xrightarrow{g}E'$, \hbox{$f_2=B\xrightarrow{f}B'\xrightarrow{\eta}B'$}, $g_2=E\xrightarrow{g}E'\xrightarrow{\eta}E'$. We have to prove that the commutative diagrams 
	\[
		\begin{diagram}
			\node{E}
				\arrow{e,t}{g_1}
				\arrow{s,l}{p}
			\node{E'}
				\arrow{s,r}{p'}
		\\
			\node{B}
				\arrow{e,t}{f_1}
			\node{B'}
		\end{diagram}
		\;\;\;\text{ and }\;\;\;
		\begin{diagram}
			\node{E}
				\arrow{e,t}{g_2}
				\arrow{s,l}{p}
			\node{E'}
				\arrow{s,r}{p'}
		\\
			\node{B}
				\arrow{e,t}{f_2}
			\node{B'}
		\end{diagram}
	\]
	induce homotopic maps between fibers, i.e.\ that maps $T_{f_1}^{g_1}$ and $T_{f_2}^{g_2}$ are homotopic. Let $F:B\times I\to B'$ be a homotopy between $f_2=F|_{B\times 0}$ and $f_1=F|_{B\times 1}$, and let $L: E\times I\to E'$ be a lifting extension fitting into the diagram
	\[
		\begin{diagram}
			\node{E\times 0}
				\arrow{e,t}{g_2}
				\arrow{s,J}
			\node{E'}
				\arrow{s,r}{p'}
		\\
			\node{E\times I}
				\arrow{e,t}{F\circ p'}
				\arrow{ne,t,..}{L}
			\node{B'}
		\end{diagram}	
	\]
	Denote $g_2'=L|_{E\times 1}:E\to E'$. Maps $T_{f_2}^{g_2}$ and $T_{f_1}^{g_2'}$ are homotopic since $(F,L)$ is a homotopy between transformations $(f_2,g_2)$ and $(f_1,g'_2)$. Finally, the diagram  
	\[
		\begin{diagram}
			\node{E}
				\arrow{e,tb,dd}{g_1}{g'_2}
				\arrow{s,l}{p}
			\node{E'}
				\arrow{s,r}{p'}
		\\
			\node{B}
				\arrow{e,t}{f_1}
			\node{B'}
		\end{diagram}
	\]	
	satisfies the assumptions of Lemma \ref{l:FibAct}, since the homomorphism
	\begin{multline*}
		\pi_1(\map(E,E')_{g_1})= \pi_1(\map(E,E')_{g\eta}) \underset{\simeq}{\xrightarrow{\map(\eta^{-1},E')_*}} \pi_1(\map(E,E')_g) \xrightarrow{\map(E,p')_*}\\
		\pi_1(\map(E,B')_{p'g})\underset{\simeq}{\xrightarrow{\map(\eta,B')_*}}\pi_1(\map(E,B')_{ p'g \eta})=\pi_1(\map(E,B')_{p'g_1})
	\end{multline*}
	is an epimorphism. Thus maps $g_1|_{F},g_2'|_{F}:F\to (p')^{-1}(f_1(b_0))$ and are homotopic. Finally we obtain a sequence of homotopic maps $F\to F'$
\[
	T_{f_1}^{g_1}=t_{f_1(b_0),b_0'}\circ g_1|_F\sim t_{f_1(b_0),b_0'}\circ g'_2|_F=T_{f_1}^{g_2'}\sim T_{f_2}^{g_2}
\]	
	which ends the proof.
\end{proof}

\subsection*{Equivalence of homotopy centralizers}

\begin{prp}\label{p:CentCrit}
	Let $G$ and $H\subseteq H'$ be compact Lie groups and let $\alpha:G\to H$ be a homomorphism. Assume that for every $p$-stubborn subgroup $P\subseteq G$ the centralizers $C_H(\alpha(P))$ and $C_{H'}(\alpha(P))$ are equal. Then the map
	\[
		\map(BG_p^\wedge,BH_p^\wedge)_{B\alpha}\to \map(BG_p^\wedge,(BH')_p^\wedge)_{B\alpha}
	\]
	is a homotopy equivalence.
\end{prp}
\begin{proof}
	Let $i:H\to H'$ be the inclusion. Denote $\beta_P:=\alpha|_P$ and
	 \[
		\map(BG_p^\wedge,BH_p^\wedge)_{(\alpha)}:=\{f:BG_p^\wedge\to BH_p^\wedge:\;\forall_{G/P\in\mathcal{R}_p(G)}\;  f|_{BP_p^\wedge}\sim (B\alpha|_P)_p^\wedge\}.
	\]
	There is a commutative diagram
	\[
		\begin{diagram}
			\node{\map(BG,BH_p^\wedge)_{(\alpha)}}
				\arrow[2]{e,t}{i_*}
				\arrow{s,lr}{\simeq}{(\ref{t:JMODecomp})}
			\node{}
			\node{\map(BG,(BH')_p^\wedge)_{(i\circ \alpha)}}
				\arrow{s,lr}{\simeq}{(\ref{t:JMODecomp})}
		\\
			\node{\map(\operatorname*{hocolim}_{G/P\in \mathcal{R}_p(G)} EG/P,BH_p^\wedge)_{(\alpha)}}
				\arrow[2]{e,t}{i_*}
				\arrow{s,l}{\simeq}
			\node{}
			\node{\map(\operatorname*{hocolim}_{G/P\in \mathcal{R}_p(G)} EG/P,(BH')_p^\wedge)_{(i\circ \alpha)}}
				\arrow{s,l}{\simeq}
		\\
			\node{\operatorname*{holim}_{G/P\in \mathcal{R}_p(G)} \map(EG/P,BH_p^\wedge)_{\alpha|_P}}
				\arrow[2]{e,t}{i_*}
			\node{}
			\node{\operatorname*{holim}_{G/P\in \mathcal{R}_p(G)} \map(EG/P,(BH')_p^\wedge)_{i\circ \alpha|_P}}
		\end{diagram}
	\]
	with horizontal arrows induced by $i:H\to H'$ and vertical ones being homotopy equivalences. The bottom horizontal map is an equivalence by assumptions, then so is the upper one. By restricting to the component of $B\alpha_p^\wedge$ at left-hand side, and of $(Bi\circ\alpha)_p^\wedge$ at right-hand side we obtain the conclusion.
\end{proof}

\begin{prp}\label{p:CentrPQUU}
	Let $k,l$ be positive integers. If $P\subseteq U(k)$, $Q\subseteq U(l)$ are $p$-stubborn subgroups, then
	\[C_{U(k+l)}(P\times Q)=C_{U(k)\times U(l)}(P\times Q)=C_{U(k)}(P)\times C_{U(l)}(Q).\]
\end{prp}
\begin{proof}
	Let $\iota_P:P\subseteq U(k)$, $\iota_Q:Q\subseteq U(l)$ denote the inclusions and let $\pi_P:P\times Q\to P$, $\pi_Q:P\times Q\to Q$ be projections. The representation $\iota_P$ do not contain a trivial factor --- otherwise the normalizer of $P$ in $U(k)$ would have a greater dimension than $P$ which contradicts $p$-stubbornness. Every irreducible representation of $P\times Q$ has the form $\varrho\bar\otimes\sigma:=(\varrho\circ\pi_P)\otimes (\sigma\circ \pi_Q)$, where $\varrho$ and $\sigma$ are irreducible representations of $P$ and $Q$ respectively. If $\iota_P\circ \pi_P$ contains a subrepresentation isomorphic to $\varrho\bar\otimes\sigma$, then $\sigma$ is  trivial and $\varrho$ is non-trivial; if it is contained in $\iota_Q\circ \pi_Q$ then $\varrho$ is non-trivial. As a consequence, there exists no irreducible representation of $P\times Q$ which is contained as a summand in both $\iota_P\circ \pi_P$ and $\iota_Q\circ \pi_Q$. By Schur's lemma we obtain
	\begin{multline*}
		C_{U(k+l)}(P\times Q)\cong C_{U(k+l)}((\iota_P\circ \pi_P)\oplus (\iota_Q\circ \pi_Q))
		\cong\\
		 C_{U(k)}(\iota_P)\times C_{U(l)}(\iota_Q) \cong C_{U(k)}(P)\times C_{U(l)}(Q). \qedhere
	\end{multline*}
\end{proof}

\begin{prp}\label{p:CentrOfJJ}
	Let $N$ be a finite group and $P$ a $p$-toral h-$N$-group. Assume that  $\alpha:P\to U(d)$ and $\beta:P\to U(d')$ are h-$N$-invariant homomorphisms. Then the map
	\[
		\map(\mathfrak{J}_{\alpha}\times \mathfrak{J}_{\beta},\mathfrak{J}_{\alpha}\times \mathfrak{J}_{\beta})_{id}
			\to
		\map(\mathfrak{J}_{\alpha}\times \mathfrak{J}_{\beta},\mathfrak{J}_{\alpha\oplus\beta})_{B(\subseteq)_p^\wedge}
	\]
	induced by the inclusion $J_\alpha\times J_\beta\subseteq J_{\alpha\oplus\beta}$ is a homotopy equivalence.
\end{prp}
\begin{proof}
	Let $Q$ be a $p$-stubborn subgroup of $J_{\alpha}\times J_\beta$. By \cite[Th. 1.6]{JMO}, every $p$-stubborn subgroup of $J_{\alpha}\times J_\beta$ has the form
	\[
		\prod_{\varrho\in\IR(P)} Q_\varrho \times \prod_{\varrho\in\IR(P)} Q'_\varrho\subseteq \prod_{\varrho\in\IR(P)} U(c_\alpha^\varrho)\times \prod_{\varrho\in\IR(P)} U(c_\beta^\varrho),
	\]
	where $Q_\varrho\subseteq U(c_\alpha^\varrho)$, $Q'_\varrho\subseteq U(c_\beta^\varrho)$ are $p$-stubborn subgroups. Then
	\begin{multline*}
		C_{J_{\alpha}(P)\times J_\beta(P)}(Q)=\prod_{\varrho\in\IR(P)}C_{U(c_\alpha^\varrho)}(Q_\varrho)\times \prod_{\varrho\in\IR(P)}C_{U(c_\beta^\varrho)}(Q'_\varrho)\cong\\
		\prod_{\varrho\in\IR(P)} C_{U(c_\alpha^\varrho)\times U(c_\beta^\varrho)}(Q_\varrho\times Q'_\varrho)
		\buildrel{(\ref{p:CentrPQUU})}\over\cong 
		\prod_{\varrho\in\IR(P)} C_{U(c_\alpha^\varrho+c_\beta^\varrho)}(Q_\varrho\times Q'_\varrho)=\\
		C_{J_{\alpha+\beta}(P)}(Q).
	\end{multline*}
	The conclusion follows by \ref{p:CentCrit}.
\end{proof}

\subsection*{Equivariant equivalence of fibers}
Let $G$ be a compact connected Lie group, $P\subseteq G$ a $p$-toral subgroup, and $N=N_G(P)/P$. Fix h-$N$-invariant homomorphisms $\alpha:P\to U(d)$ and $\beta:P\to U(d')$. Let $\mathfrak{F}_{\alpha,\beta}$ be the fiber of the $N$-equivariant fibration $m^G_{\alpha,\beta}:\mathfrak{M}^G_\alpha\widetilde{\times}\mathfrak{M}^G_\beta\to \mathfrak{M}^G_{\alpha\oplus\beta}$ (cf.\ \ref{e:FibrM}). Since both $\mathfrak{M}^G_{\alpha\oplus\beta}$ and $\mathfrak{F}_{\alpha,\beta}$ are simply connected, $\mathfrak{F}_{\alpha,\beta}$ carries an induced right h-$N$-action.

\begin{prp}\label{p:EquivFG}
	The right h-$N$-spaces $\mathfrak{F}_{\alpha,\beta}$ and $\mathfrak{G}_{\alpha,\beta}$ (cf.\ \ref{e:Grass}) are equivalent.
\end{prp}
\begin{proof}
	By \ref{p:EquivFibr} the fibrations $j_{\alpha,\beta}$ and $m_{\alpha,\beta}^G$ are h-$N$-equivalent. Thus the map $\mathfrak{G}_{\alpha,\beta}\to\mathfrak{F}_{\alpha,\beta}$ induced by an h-$N$-equivalence of fibrations is a homotopy equivalence, and \ref{p:CentrOfJJ} implies that the assumptions of \ref{p:hNActCrit} are satisfied. Then it is also an h-$N$-equivalence.
\end{proof}

\begin{proof}[Proof of {Theorem \ref{t:hNEq}}]
	If $\alpha:G\to U(d)$, $\beta:G\to U(d')$ are homomorphisms representing $\mu_P$ and $\nu_P$ respectively, and $N=N_G(P)/P$, then there is an $N$-homeomorphism
	\[
		\Fib_{\mu,\nu}^p(G/P)\cong \fF_{\alpha,\beta}.
	\]
	The conclusion follows from \ref{p:EquivFG}.
\end{proof}

\end{document}